\documentclass[reqno]{amsart}
\usepackage{graphicx, color}
\usepackage{dsfont}
\usepackage{ragged2e}
\usepackage{tikz}
\usetikzlibrary{patterns}
\usetikzlibrary{backgrounds}
\usepackage{amsmath,amsfonts, amssymb}
\usepackage{wrapfig}
\usepackage{paralist}
\usepackage{graphics} 
  \usepackage{epsfig} 
 \usepackage[colorlinks=true]{hyperref}
 \usepackage{graphicx}  \usepackage{epstopdf}
 \usepackage{bbm}
\hypersetup{urlcolor=blue, citecolor=red}

\usepackage{ mathrsfs }

\DeclareMathOperator*{\esssup}{ess\,sup}

\newtheorem{theorem}{Theorem}[section]
\newtheorem{corollary}{Corollary}

\newtheorem{lemma}[theorem]{Lemma}
\newtheorem{proposition}{Proposition}[section]

\theoremstyle{definition}

\newcommand{\japvi}{\left\langle v \right\rangle_i}
\newcommand{\japvj}{\left\langle v \right\rangle_j}
\newcommand{\japyj}{\left\langle y \right\rangle_j}

\newcommand{\japvip}{\left\langle v' \right\rangle_i}
\newcommand{\japvjps}{\left\langle v'_* \right\rangle_j}

\newcommand{\lk}{L_k^1}

\newcommand{\lpki}{L_{k,i}^p}

\newcommand{\lnula}{L_0^1}

\newcommand{\ldva}{L_2^1}
\newcommand{\ldvak}{L^2_k}

\newcommand{\ljedk}{L_k^1}
\newcommand{\lnminus}{L^1_{N-2}}

\newcommand{\lpk}{L_k^p}

\newcommand{\lpnula}{L_0^p}

\newcommand{\lplam}{L_\lambda^p}
\newcommand{\ldvanula}{L_0^2}

\newcommand{\lrnula}{L_0^r}

\newcommand{\lqlam}{L_\lambda^q}

\newcommand{\linfnula}{L_0^\infty}
\newcommand{\lpcomp}{p}
\newcommand{\lqcomp}{q}
\newcommand{\lrcomp}{r}
\newcommand{\lrpcomp}{r'}
\newcommand{\linfcomp}{\infty}
\newcommand{\lonecomp}{1}
\newcommand{\lppcomp}{p'}

\newcommand{\linfp}{\infty,+}
\newcommand{\lrp}{r,+}
\newcommand{\lpp}{p,+}
\newcommand{\lqp}{q,+}

\newcommand{\F}{\mathbb{F}}
\newcommand{\G}{\mathbb{G}}

\newcommand{\Q}{\mathbb{Q}(\F,\F)}
\newcommand{\QFG}{\mathbb{Q}(\F,\G)}
\newcommand{\QFGp}{\mathbb{Q}^+(\F,\G)}
\newcommand{\QFGm}{\mathbb{Q}^-(\F,\G)}

\newcommand{\QFGptil}{\mathbb{Q}^+(\tilde{\F},\tilde{\G})}
\newcommand{\Qp}{\mathbb{Q}^+(\F, \F)}
\newcommand{\Qm}{\mathbb{Q}^-(\F, \F)}
\newcommand{\QQp}{\mathbb{Q}^+}
\newcommand{\Qpone}{\mathbb{Q}^+_1(\F, \F)}
\newcommand{\Qpinf}{\mathbb{Q}^+_\infty(\F, \F)}

\newcommand{\nFlpk}{\left\|\mathbb{F}\right\|_{\lpk}}

\newcommand{\md}{\mathrm{d}}

\newcommand{\bb}{\left\| b_{ij} \right\|_{L^\infty(\mathbb{S}^{N-1})} }
\newcommand{\bbjed}{\left\| b_{ij} \right\|_{L^1(\mathbb{S}^{N-1};\mathrm{d}\sigma)} }
\newcommand{\bbdvaepsjed}{\left\| b_{ij}^{\varepsilon_{ij},2} \right\|_{L^1(\mathrm{d}\sigma)} }

\newcommand{\E}{E_{v v'}}
\newcommand{\p}{P_{r_{ij}}(v,v')}
\newcommand{\K}{K_i[\G](v,x)}
\newcommand{\Kg}{K_{\gamma_{ij},i}[\G](v,x)}
\newcommand{\Kn}{K_{0,i}[\G](v,x)}
\newcommand{\KNmdva}{K_{N-2,i}[\G](v,x)}
\newcommand{\Qpg}{Q^+_{\gamma_{ij},ij}}
\newcommand{\Qpn}{Q^+_{0,ij}}
\newcommand{\QpNmdva}{Q^+_{N-2,ij}}

\newcommand{\ko}{k_0}
\newcommand{\kN}{k_N}

\newcommand{\opP}{\mathscr{P}_{ij}(\varphi, \chi)}
\newcommand{\opPs}{\mathscr{P}_{ij}(\varphi_p^*, \chi_q^*)}
\newcommand{\opPsinf}{\mathscr{P}_{ij}(\varphi_\infty^*, \chi_\infty^*)}
\newcommand{\opB}{\mathscr{B}_{ij}(\varphi, \chi)}
\newcommand{\opBs}{\mathscr{B}_{ij}(\widetilde{\varphi_p^*},\widetilde{\chi_q^*})}

\newcommand{\CopPr}{C_{p,q,r}^{\mathscr{P}_{ij}}}
\newcommand{\CopPinf}{C_{\infty}^{\mathscr{P}_{ij}}}
\newcommand{\CopPinftri}{C_{\infty,\infty,\infty}^{\mathscr{P}_{ij}}}

\newcommand{\CopBr}{C_{p,q,r}^{\mathscr{B}_{ij}}}
\newcommand{\CopBinf}{C_{\infty}^{\mathscr{B}_{ij}}}
\newcommand{\CopBinftri}{C_{\infty,\infty,\infty}^{{\mathscr{B}}_{ij}}}

\newcommand{\CopQr}{C_{p,q,r}^{\QQp}}
\newcommand{\CopQinf}{C_{\infty}^{\QQp}}
\newcommand{\CopQone}{C^{\QQp}_{1}}
\newcommand{\CopQonetri}{C^{\QQp}_{1,1,1}}

\newcommand{\CijQp}{C^{ij}_{p,q,r}}
\newcommand{\CijQppppinf}{C^{ij}_{p,p',\infty}}

\newcommand{\Cijtwo}{C^{\mathscr{P}_{ij}}_{r',\infty,r'}}
\newcommand{\Cijthree}{C^{\mathscr{P}_{ij}}_{\infty,r',r'}}
\newcommand{\Cijfive}{C^{\mathscr{P}_{ij}}_{\infty,1,1}}
\newcommand{\Cijsix}{C^{\mathscr{P}_{ij}}_{1,\infty,1}}
\newcommand{\Cijfour}{\left(C^{\mathscr{P}_{ij}}_{1,\infty,1}\right)^{1/p}\left(C^{\mathscr{P}_{ij}}_{\infty,1,1}\right)^{1/p'} }
\newcommand{\Cijroner}{C^{\QQp}_{p,1,p}}

\newcommand{\mdx}{\md \xi_N^{b_{ij}}(s)}

\newcommand{\gij}{\gamma_{ij}}
\newcommand{\lij}{{\lambda_{ij}}}
\newcommand{\Mij}{ M_{ij}}
\newcommand{\clb}{c_{lb}}
\newcommand{\ct}{\tilde{c}}

\newcommand{\expLoneLpFzero}{\left\| \F_0 e^{\alpha_0 \langle \cdot \rangle^{s}}\right\|_{(\lnula \cap \lpnula)}}
\newcommand{\expLpF}{\left\| \F e^{\alpha \langle \cdot \rangle^{s}}\right\|_{\lpnula}}
\newcommand{\expLoneF}{\left\| \F e^{\tilde{\alpha} \langle \cdot \rangle^{s}}\right\|_{\lnula}}

\newcommand{\expLinfF}{\left\| \F e^{\alpha \langle \cdot \rangle^{s}}\right\|_{\linfnula}}

\newcommand{\bone}{b_{ij}^{\varepsilon_{ij},1}}
\newcommand{\btwo}{b_{ij}^{\varepsilon_{ij},2}}

\title[]
      {Propagation of $L^p_{\beta}$-norm, $ 1<p\le \infty$, for the system of Boltzmann equations for monatomic gas mixtures}

\author[Erica de la Canal, Irene M. Gamba and Milana Pavi\'c-\v Coli\'{c}]{}

\begin{document}
\maketitle


\centerline{\scshape Erica de la Canal}
\medskip
{\footnotesize
	\centerline{Department of Mathematics}
	\centerline{University of Texas at Austin}
	\centerline{2515 Speedway Stop C1200
		Austin, Texas 78712-1202}
	
}

\medskip

\centerline{\scshape Irene M. Gamba}
\medskip
{\footnotesize
	\centerline{Department of Mathematics}
	\centerline{University of Texas at Austin}
	\centerline{2515 Speedway Stop C1200
		Austin, Texas 78712-1202}
	
}

\medskip

\centerline{\scshape Milana Pavi\'c-\v Coli\'{c}}
\medskip
{\footnotesize
  \centerline{Department of Mathematics and Informatics}
  \centerline{Faculty of Sciences, University of Novi Sad}
 \centerline{Trg Dositeja Obradovi\'ca 4, 21000 Novi Sad, Serbia}
}

\bigskip

\centerline{\today}

\bigskip


\begin{abstract} \justify
With the existence and uniqueness of a vector value solution for the full non-linear homogeneous Boltzmann system of equations describing multi-component monatomic gas mixtures for binary interactions proved \cite{IG-P-C}, we present in this manuscript several properties for such a solution. We start by proving the gain of integrability of the gain term of the multispecies collision operator, extending the work done previously for the single species case \cite{IG-Alonso}. In addition, we study the integrability properties of the multispecies collision operator as a bilinear form, revisiting and expanding the work done for a single gas \cite{IG-Alonso-Carneiro}. With these estimates, together with a control by below for the loss term of the collision operator as in \cite{IG-P-C}, we develop the propagation for the polynomially and exponentially  $\beta$-weighted $L^p_\beta$-norms for the vector value solution.  Finally, we extend such    $L^p_\beta$-norms propagation property to $p=\infty$.

\end{abstract}


\tableofcontents

\section{Introduction}

We consider a mixture of $I$ monatomic gases in a space homogeneous setting. Each component of the mixture can be statistically described by its distribution function $f_i:=f_i(t,v)$, depending on time $t \geq 0$ and velocity of molecules $v \in \mathbb{R}^N$, that changes due to binary interactions with other particles of the same or different species. For an $i$ fixed, each $f_i$ solves a Boltzmann type equation where the collision operator takes into account not only the influence of particles of the same species, but all other species. Since we are considering all species simultaneously, we introduce a vector valued set of distribution functions $\F := [f_i]_{1\leq i \leq I}$, whose change due to binary collisions of particles is expressed by a vector of collision operators, with its $i-$th component given by $[\mathbb{Q}(\F)]_{i}:= \sum_{j=1}^I Q_{ij}(f_i, f_j)$. Then, the evolution of a mixture is leaded by  a system of Boltzmann equations.\\

The existence and uniqueness of the solution for the non-linear system of spatially homogeneous Boltzmann equations for multi-species mixtures with binary interactions has been proven recently in \cite{IG-P-C}, in a vector valued Banach space with a norm depending on the species mass fractions, to be defined in the next sections. That result is obtained by following general ODE theory in Banach spaces that studies the evolution of such vector valued solutions with their suitable norm, without requesting entropy boundedness. Such normed spaces provided estimates to rigorously prove the generation and propagation of scalar polynomial and exponential moments of the vector valued $\F$.   \\

The innovative tools and results at the core of this manuscript are the following: first the introduction of suitable $L^p_\beta$ vector valued spaces whose norms have dependence on specific weights depending on the species mass fractions and a new explicit Carleman integral representation for the positive part of the collisional operator, associated to a binary interaction of two different species with different masses. Next, by means of the new Carleman representation and the vector valued  $L^p_\beta$-spaces associated to the multi-species system, we obtain a gain of integrability estimate for the positive contribution  of the collisional form that provides explicit constants rates. Last, we show that gain of integrability estimates prove the propagation of the $L^p$ norms with polynomial and exponential weights of the vector valued solution of the Boltzmann system of equations for mixtures.

%

The techniques used in this manuscript are extensions or adaptations of results developed for the scalar Boltzmann equations. But since  in the mixture framework each component of the mixture is characterized by the molecular mass $m_i$, the symmetry properties of the binary collisions are no longer valid, which yields to changes in the mathematical treatment.\\

It is noteworthy that the gain of integrability estimates essentially follows, after the generalization to the multi-species problem, the strategy devised in \cite{IG-Alonso}, which consists in first showing the control of a weighted $L^2$ norm of the positive part of the collision operator of each species by a lower order weighted $L^2$ norm of the input vector value function in a suitable Banach space. After the  $L^2$ control, one can obtain a Young's type inequality for the gain term of the collision operator for mixture of gases by means of a  weak form of the positive term of the collisional operator as an extension of the original strategy developed in \cite{IG-Alonso-Carneiro} and a subsequent interpolation argument for the system. 

 This  extends the work in \cite{IG-Alonso-Carneiro} and \cite{IG-Alonso}, both proven for a single species Boltzmann equation. With this two estimates, we are able to state and prove the gain of $L^p$ integrability of the positive part of the collision operator $1<p<\infty$. In addition, we prove the propagation of $L^\infty$ norms, following the same approach as in \cite{IG-Alonso-Taskovic} applied to  the single component gas.

The paper is organized as follows. In section \ref{Preliminaries} we describe the kinetic model along with the notation used through the manuscript and state the main results of this work. In the same section we define the suitable Banach spaces needed to get our results and their notation and we state different ways to write the gain part of the collisional operator, both in the classical definition and the weak formulation. Then, in Section \ref{Section Ltwo} and \ref{Section Lp} we prove the gain of integrability for $L^2$ and $L^p$ polynomially weighted norms respectively. In section \ref{Section propagation} we prove the main results: the propagation of $L^p$ norms, both with polynomial and exponential weights, and in Section \ref{Section Linfty}  we extend this results for the case $p= \infty$. Finally, there is an Appendix with some calculations and statements needed.

\section{Preliminaries and Main results} \label{Preliminaries}
\subsection{The kinetic model for monatomic gas mixtures } \label{Section kinetic model}

Each component of the mixture, namely  $\mathcal{A}_i$ with $1\leq i \leq I$, is described with its own distribution function $f_i:=f_i(t, v)\geq 0$, that depends on time $t>0$ and particle velocity   $v \in \mathbb{R}^N$. In order to describe its evolution, we first model the  binary interaction of two colliding molecules.

\subsubsection{Collision process}
We fix two molecules of species $\mathcal{A}_i$ and $\mathcal{A}_j$ that are going to collide. Let molecule of species $\mathcal{A}_i$ have mass $m_i$ and velocity $v'$; and molecule of species $\mathcal{A}_j$,  mass $m_j$ and velocity $v'_*$ before the collision. After the collision, they belong to the same species, have the same mass, but their velocities changed to $v$ and $v_*$, respectively. During the elastic collision, conservation of momentum and kinetic energy hold, that is
\begin{equation}\label{CL micro}
\begin{split}
m_i v' + m_j v'_* &= m_i v + m_j v_*, \\
m_i \left| v' \right|^2 +  m_j \left| v'_* \right|^2 &= m_i \left| v \right|^2 +  m_j \left| v_* \right|^2.
\end{split}
\end{equation} 
Let $r_{ij}\in (0,1)$ be the mass contribution of the molecule of species $\mathcal{A}_i$ to the sum of masses of two colliding molecules of species $\mathcal{A}_i$ and $\mathcal{A}_j$, i.e. denote
\begin{equation}\label{rij}
r_{ij}:= \frac{m_i}{m_i + m_j} \quad \Rightarrow \quad r_{ji} :=  1-r_{ij}=\frac{m_j}{m_i + m_j}.
\end{equation}
Then equations \eqref{CL micro} can be parametrized with a parameter $\sigma \in \mathbb{S}^{N-1}$,
so that pre-collisional quantities can be written in terms of post-collisional ones as
\begin{equation}\label{collision transformation}
v' = v + (1-r_{ij}) (\left|u\right| \sigma - u ), \qquad v'_* = v_* -r_{ij}  (\left|u\right| \sigma - u ),
\end{equation}
where $u := v-v_*$ is the relative velocity. In other words,  $\sigma$ is in the direction of the pre-collisional relative velocity $u'=v'-v'_*$,
\begin{equation*}
u' = \left|u\right| \sigma.
\end{equation*} 

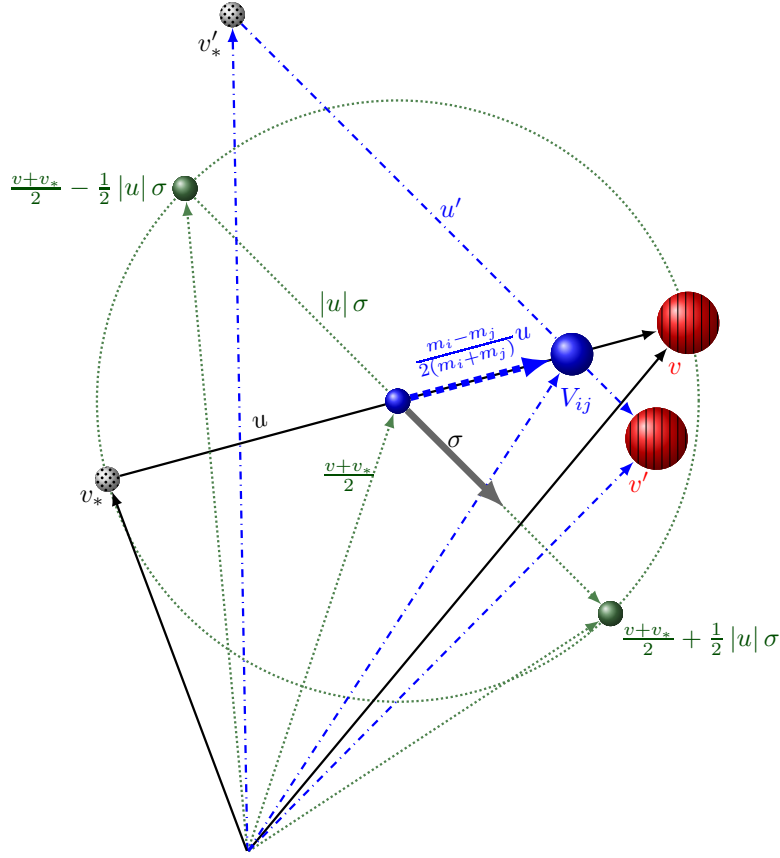
\begin{figure}
	\begin{center}
		\begin{tikzpicture}[scale=2]
		\node (v) at (1.93,0.52) [preaction={shade, ball color=red},pattern=vertical lines, pattern color=black,circle,scale=2.5] {};
		\node (vs) at (-1.93,-0.52) [preaction={shade, ball color=white!70!gray},circle,scale=1, ,pattern=crosshatch dots, pattern color=black] {};
		\node (V) at (1.158,0.312)  [circle,scale=1.7, shade,ball color=blue] {};
		\node (vp) at (1.72,-0.25)  [preaction={shade,ball color=red},circle,scale=2.5, ,pattern=vertical lines, pattern color=black] {};
		\node (vps) at (-1.1,2.57) [preaction={shade,ball color=white!70!gray}, circle,scale=1,pattern=crosshatch dots, pattern color=black ] {};
		\node (V old) at (0,0)	 [circle,scale=1, shade,ball color=blue] {};
		\node (start) at (-1,-3)  [circle,scale=0.1, gray] {};
		\node (vp old) at (1.414,-1.414)   [circle,scale=1, shade,ball color=green!30!black!70!white] {};
		\node (vps old) at (-1.414,1.414)   [circle,scale=1, shade,ball color=green!30!black!70!white] {};
		
		\begin{scope}[>=latex]
		\draw[->,line width=0.6ex,gray!80!black] (V old) -- (0.7071,-0.7071) node[midway, above,black] {$\sigma$};
		\draw[->,line width=0.2ex] (start) --  (vs) node[left,pos=0.98] {$v_*$};
		\end{scope}
		\begin{scope}[>=latex, on background layer]
		\draw[densely dotted,color=green!30!black!70!white,line width=0.2ex] (0,0) circle (2cm);
		\draw[densely dotted,->,color=green!30!black!70!white,line width=0.2ex] (vps old) -- (vp old) node[densely dotted,color=green!30!black,pos=0.3,above right=-0.1cm,line width=0.2ex] {$  \left|u \right| \sigma$};
		\draw[->,line width=0.2ex] (start) -- (v)  node[pos=0.96,right=0.1cm] {\color{red}$v$};
		\draw[->,line width=0.2ex] (vs) -- (v) node[pos=0.3,sloped, black,above left] {$u$};
		\draw[densely dotted,->,color=green!30!black!70!white,line width=0.2ex] (start) -- (V old) node[color=green!30!black,pos=0.9,below left=-0.1cm] {$\frac{v+v_*}{2}$};
		\draw[->,dashdotted,color=blue,line width=0.2ex]  (vps)--(vp) node[color=green!30!black,above right=-0.1cm,midway] {\color{blue}$u'$};
		\draw[densely dotted,->,color=green!30!black!70!white,line width=0.2ex] (start) -- (vp old) node[color=green!30!black,below right,line width=0.2ex] {$ \frac{v +  v_*}{2} + \frac{1}{2} \left|u \right| \sigma$};
		\draw[densely dotted,->,color=green!30!black!70!white,line width=0.2ex] (start) -- (vps old) node[color=green!30!black, left=0.1cm] {$ \frac{v +  v_*}{2} - \frac{1}{2} \left|u \right| \sigma$};
		\draw[->,densely dashed,blue,line width=0.6ex] (V old) --  (V) node[midway, sloped,above, blue] {$\frac{m_i-m_j}{2(m_i+m_j)}u$};
		\draw[dashdotted,->,color=blue,line width=0.2ex] (start) -- (V) node[pos=0.98,below right=-0.05cm] {\color{blue}$V_{ij}$};
		\draw[dashdotted,->,color=blue,line width=0.2ex] (start) -- (vps) node[pos=0.98,left] {\color{black}$v'_*$};
		\draw[dashdotted,->,color=blue,line width=0.2ex] (start) -- (vp) node[pos=0.98,below right=-0.1cm] {\color{red}$v'$};
		\end{scope}	
		\end{tikzpicture}
	\end{center}
	\caption{Illustration of the collision transformation, with notation $V_{ij}:=\frac{m_i v + m_j v_*}{m_i + m_j}$, $u:=v-v_*$,  $u':=v'-v'_*$. The displacement of the center of mass with respect to a single component elastic binary interaction is given by $(r_{ij}-\frac 12)u = \frac{m_i-m_j}{2(m_i+m_j)}u$, if $m_i>m_j$.  
		Solid lines denote vectors after collision, or given data. Dash-dotted vectors  represent primed (pre-collisional) quantities  that can be calculated from the given data, and compared to the case $m_i=m_j$, represented by dotted vectors. Dashed vector direction is the displacement along the direction of the relative velocity $u$ proportional to the half difference of relative masses, (which clearly vanishes for  $m_i = m_j$, reducing the model to a classical collision). Note that the scattering direction $\sigma$  is preserved as  the pre-collisional relative velocity $u'$ keeps the same magnitude as the post-collisional $u$,   $u'$ is parallel the reference elastic  pre-collisional relative velocity $|u|\sigma$. This figure is reproduced from \cite{IG-P-C}. }
	\label{picture}
\end{figure}

\subsubsection{The system of Boltzmann equations}

Since   the whole mixture is considered simultaneously, we are led to introduce a vector of distribution functions,
\begin{equation*}
\F = \left[f_i\right]_{1\leq i \leq I}.
\end{equation*}
The vector-value function $\F$ satisfies the system of Boltzmann equations,
\begin{equation}\label{BE}
\partial_t \F(t,v)=\Q(t,v), \quad t>0, \ v\in \mathbb{R}^N,
\end{equation}
where $\Q$ is a vector of multispecies collision operators whose $i-$th component is 
\begin{equation}\label{Boltzmann operator}
\left[\Q\right]_{i} = \sum_{j=1}^I Q_{ij}(f_i, f_j)(v),
\end{equation}
and $Q_{ij}$ is defined below.
\subsection{Pairwise collision operator $Q_{ij}$}
Let $f$ be the distribution function for the species $\mathcal{A}_i$ and let the distribution function $g$  be associated to the species  $\mathcal{A}_j$. The pairwise collision operator describing  the collisions of molecules of species $\mathcal{A}_i$ with the molecules of species $\mathcal{A}_j$ is defined as
\begin{equation}\label{collision operator}
Q_{ij}(f,g)(v)= \int_{\mathbb{R}^N} \int_{\mathbb{S}^{N-1}} \left( f(v') \, g(v'_*) - f(v) \, g(v_*)  \right) \,\mathcal{B}_{ij}(\left|u\right|,\hat{u}\cdot \sigma) \, \mathrm{d} \sigma \, \mathrm{d}v_*,
\end{equation}
where the pre-collisional velocities  are given by \eqref{collision transformation}, and $\hat{u}:=u/\left|u\right|$.

The other way around, we define the pairwise collision operator that describes collision of molecules of species $\mathcal{A}_j$ with the ones of species $\mathcal{A}_i$ 
\begin{equation}\label{collision operator ji}
Q_{ji}(g,f)(v)= \int_{\mathbb{R}^N} \int_{\mathbb{S}^{N-1}} \left( g(w') \, f(w'_*) - g(v) \, f(v_*)  \right) \,\mathcal{B}_{ji}(\left|u\right|,\hat{u}\cdot \sigma) \, \mathrm{d} \sigma \, \mathrm{d}v_*,
\end{equation}
where now velocities $w' $ and $w'_*$ differ from \eqref{collision transformation} by mass interchange $m_i \leftrightarrow m_j$,
\begin{equation}\label{collision transformation ji}
\begin{split}
w' &= v + (1-r_{ji}) (\left|u\right| \sigma - u ) = v + r_{ij} (\left|u\right| \sigma - u ) , \\
 w'_* &= v_* -r_{ji}  (\left|u\right| \sigma - u )  = v_* -(1-r_{ij})  (\left|u\right| \sigma - u ),
\end{split}
\end{equation}
after noting that $r_{ji} = 1-r_{ij}$ from definition \eqref{rij}.

The collision kernel associated to transition probabilities of exchanging states at an interaction,  $\mathcal{B}_{ij}$, is a positive a.e. measure  
 that satisfies the  micro-reversibility assumption given by the invariance  between the switching of   post- and pre-collisional velocities for a pair of interacting particles. 
That is, $\mathcal{B}_{ij}$ remains invariant  under the following exchange $(v,v_*,\sigma)\leftrightarrow (v',v'_*,\sigma')$, with $\sigma'=u/ \left|u\right|$, and $(v,v_*,\sigma)\leftrightarrow (v_*,v,-\sigma)$, resulting in the following identity 
\begin{equation}\label{cross-sym}
\mathcal{B}_{ij}(\left|u\right|,\hat{u}\cdot \sigma) = \mathcal{B}_{ij}(\left|u'\right|,\hat{u}'\cdot \sigma')=\mathcal{B}_{ji}(\left|u\right|,\hat{u}\cdot \sigma).
\end{equation}

In our scope, we only deal with the hard potential and integrable angular transition probability case  for each particle pair  $\mathcal{A}_i$ and $\mathcal{A}_j$, $1\leq i, j \leq I$. 
That means, the  terms $\mathcal{B}_{ij}$, $i,j=1,\dots,I$ are assumed to take the following form 
\begin{equation}\label{cross section}
\mathcal{B}_{ij}(\left|u\right|,\hat{u}\cdot \sigma) =\left|u\right|^{\gamma_{ij}} \, b_{ij}(\hat{u} \cdot \sigma), \ \gamma_{ij}\in (0,1], \ \text{and} \ b_{ij}(\hat{u} \cdot \sigma) \in L^1(\mathbb{S}^{N-1}),
\end{equation}
where $b_{ij}(\hat{u} \cdot \sigma)$ is the angular transition rate. We assume that each $b_{ij}$ has been symmetrized with respect to the polar angle $\theta$, and therefore its support lays in $[0,1]$.
\medskip

We note that this introductory part does not need the split in \eqref{cross section}. The general form of cross section $ \mathcal{B}_{ij}$ with the symmetries from \eqref{cross-sym} is enough to develop the new Carleman representation, to be shown in next sections. However, existence and uniqueness theory is built in \cite{IG-P-C} for the cross section \eqref{cross section}.

\subsubsection{Bilinear form of collision operator in the vector form notation} Once we defined the pairwise collision operator $Q_{ij}(f,g)$, we can introduce the vector value bilinear form of  multispecies collision operator $\QFG$. Let $\F$ and $\G$ be vectors of distribution functions $\F= \left[ f_i \right]_{i=1}^I$ and $\G= \left[ g_i \right]_{i=1}^I$. Then the collision operator associated to the distribution functions $\F$ and $\G$ is defined through its $i-$th component by
\begin{equation*}
\left[\QFG\right]_i(v):=  \sum_{j=1}^I Q_{ij}(f_i,g_j)(v), \qquad i=1,\dots,I.
\end{equation*}
Clearly, the Boltzmann operator \eqref{Boltzmann operator} is obtained for $\F=\G$.

\subsubsection{Gain and loss terms} 
When it is possible to separate the collision operator \eqref{collision operator} into the sum of  two operators (typical situation is the cut-off regime when the angular part of the cross section \eqref{cross section} is integrable), we are led to define the gain and the loss term.
Namely,  the first part of the collision operator \eqref{collision operator} is called the gain term, 
\begin{equation}\label{collision operator gain}
Q^{+}_{ij}(f,g)(v) = \int_{\mathbb{R}^N} \int_{\mathbb{S}^{N-1}} f(v') \, g(v'_*) \, \mathcal{B}_{ij}(\left|u\right|,\hat{u}\cdot \sigma) \, \mathrm{d} \sigma \, \mathrm{d}v_*,
\end{equation}
while the second part is called the loss term,
\begin{equation*}
Q_{ij}^-(f,g)(v)= f(v) \int_{\mathbb{R}^N} \int_{\mathbb{S}^{N-1}}  g(v_*)   \,\mathcal{B}_{ij}(\left|u\right|,\hat{u}\cdot \sigma) \, \mathrm{d} \sigma \, \mathrm{d}v_*,
\end{equation*}
so that \eqref{collision operator} can be rewritten as the difference of this two operators
\begin{equation*}
Q_{ij}(f,g)(v)= Q^+_{ij}(f,g)(v) - Q^-_{ij}(f,g)(v).
\end{equation*}
Passing to the vector notation, we define the vector value gain $\QFGp$ and loss $\QFGm$ terms. Namely, with the pairwise operators $Q^+_{ij}$ and $Q^-_{ij}$ defined above, the gain term $\QFGp$ is  defined as
\begin{equation}\label{gain term vector}
\left[\QFGp \right]_i= \sum_{j=1}^I Q_{ij}^+ (f_i, g_j) =  \sum_{j=1}^I  \int_{\mathbb{R}^N} \int_{\mathbb{S}^{N-1}} f_i(v') \, g_j(v'_*) \, \mathcal{B}_{ij}(\left|u\right|,\hat{u}\cdot \sigma) \, \mathrm{d} \sigma \, \mathrm{d}v_*,
\end{equation}
while the loss term $\QFGm$ is defined with 
\begin{equation*}
\left[\QFGm \right]_i= \sum_{j=1}^I Q_{ij}^- (f_i, g_j) = f_i(v) \sum_{j=1}^I   \int_{\mathbb{R}^N} \int_{\mathbb{S}^{N-1}}   g_j(v_*) \, \mathcal{B}_{ij}(\left|u\right|,\hat{u}\cdot \sigma) \, \mathrm{d} \sigma \, \mathrm{d}v_* ,
\end{equation*}
with $u=v-v_*$, $\hat{u}=u/\left|u\right|$.

\subsection{Functional spaces and integral representations for binary interactions of mixtures of monatomic gases} 

We will study the gain operator and introduce some essential concepts and notation for later studies.

\subsubsection{Functional spaces} \label{Section Functional spaces}
We are working in Banach spaces associated to the mixture, as well as to its components. Therefore we need to define the associated vector valued $L^p$ weighted spaces ($L^p_\beta$), where $\beta$ is  either a polynomial weight of order $k$ (where we will denote $\beta = k$), i.e.   $\beta\equiv \japvi^{k}$, or an exponential weight with rate $\alpha >0$ and order $s$,  for $0< s \le 1$, i.e. $\beta\equiv\text{exp}(\alpha\japvi^{s})$,   and their respective norms, where both weights depend on a renormalized mass $m_i$ for each specie mass density $f_i$ with $1\le i\le I$. 

For the case where $\beta$ is a polynomial weight of order $k$, the notation is drawn from the previous related  work  defined for the $L^1_k$ vector value functional space case for the system of Boltzmann equation for mixtures \cite{IG-P-C}, that  we extend here to $1< p\le \infty$.   
We point out that in the case of $L^1_k$-spaces these polynomial weighted norms depend on time and can be  viewed describing the time evolution of observables (or $L^1_k$-moments) associated to the vector valued probability densities. In addition, such evolution of norms was crucial to obtain a set of ordinary differential inequalities that enabled to show that  $L^1_{k,i}$-moments of each species were generated and propagated uniformly in time depending on the initial data, as much as to show these  $L^1_{k,i}$-moments were summable in $k$, for $k>k^*$, with $k^*$ a constant given in \eqref{k star} depending on $b_{ij}$ and $r_{ij}$, to obtain uniform in time estimates of exponential moments. 

Such  time dependent norms  are obtained from the weak formulation associated to the  system, which will be defined as follows.

Following the structural form of the $L^1_k$ Banach space introduced for the existence and uniqueness of the vector valued solutions for the Boltzmann system introduced in \cite{IG-P-C} we define the corresponding $L^p_k$ space 
\begin{equation*}
\lpk := \left\{ \F=\left[f_i\right]_{1\leq i \leq I}:  \sum_{i=1}^I \int_{\mathbb{R}^N} \left( \left|f_i(v) \right| \japvi^{k} \right)^p \mathrm{d}v< \infty, \ k\geq 0, 1 \leq p < \infty \right\},
\end{equation*}
Following the introduction of $L^1_k$ norms in \cite{IG-P-C}, we recall their definition of the Lebesgue weight $\japvi :=\sqrt{ 1+ \frac{m_i}{\sum_{i=1}^I m_i} \left| v \right|^2} $. We remark that the renormalization of the weight is a sufficient condition to obtain the energy identity \cite[Lemma 4]{IG-P-C}, which is essential to obtain Povzner estimates and propagation of $L^1_k$ norms.

The associated norm is then
\begin{equation}\label{norm definition}
\nFlpk := \left( \sum_{i=1}^I \int_{\mathbb{R}^N} \left( \left|f_i(v) \right|  \japvi^{k} \right)^p \mathrm{d}v \right)^{1/p}.
\end{equation}
For $p=\infty$, we define 
\begin{equation*}
L^\infty_k := \left\{ \F=\left[f_i\right]_{1\leq i \leq I}:  \sum_{i=1}^I \esssup_{v \in \mathbb{R}^N} \left( |f_i(v)|\japvi^{k} \right)< \infty, \ k\geq 0 \right\},
\end{equation*}
with its associated norm
\begin{equation*}
\| \F \|_{L^\infty_k} := \sum_{i=1}^I \esssup_{v \in \mathbb{R}^N} \left( |f_i(v)|\japvi^{k} \right).
\end{equation*}

We will also work in the components framework. We define the space of each mixture component as
\begin{equation*}
\lpki := \left\{ g: \int_{\mathbb{R}^N}\left( \left|g(v)\right|  \japvi^{k} \right)^p \mathrm{d}v < \infty, k\geq 0, 1\leq p < \infty \right\}, 
\end{equation*}
together with its norm
\begin{equation*} 
\left\| g \right\|_{\lpki} := \left( \int_{\mathbb{R}^N} \left( \left|g(v)\right| \japvi^{k} \right)^p \mathrm{d}v \right)^{1/p}.
\end{equation*}
When $p=\infty$ the space related to the specie $A_i$ is
\begin{equation*}
L^\infty_{k,i}:= \left\{ g: \esssup_{v \in \mathbb{R}^N} \left( |g(v)|\japvi^{k} \right)< \infty, \ k\geq 0 \right\}, 
\end{equation*}
with its norm
\begin{equation*}
\|g\|_{L^\infty_{k,i}} :=\esssup_{v \in \mathbb{R}^N} \left( |g(v)|\japvi^{k} \right).
\end{equation*}
Note that the norm of $\F$ in $\lpk$ is related to the norm of its components $f_i$ in the space $\lpki$ via 
\begin{equation*}
\left\| \F \right\|_{\lpk}^p = \sum_{i=1}^{I} \left\| f_i \right\|_{\lpki}^p, \quad \left\| \F \right\|_{L^\infty_{k}} = \sum_{i=1}^{I} \left\| f_i \right\|_{L^\infty_{k,i}}.
\end{equation*}
In the special case when $k=0$, we introduce the following notation for the norm
\begin{equation*}
\|g\|_{p}:= \|g\|_{L^p_{0,i}}, \quad 1\leq p \leq \infty,
\end{equation*}
for any $i=1,\dots,I$.

\

Now, when $\beta$ is an exponential weight with rate $\alpha >0$ and order $s$,  for $0< s \le 1$ we define the norm as
\begin{equation}
 \expLpF := \left( \sum_{i=1}^I \left\| f_i (t, \cdot) e^{\alpha \langle \cdot \rangle_i^s}\right\|^p_p \right)^{\frac{1}{p}} .
\end{equation}

\subsubsection{Weak form of the gain term}

Weak formulation of the pairwise collision operator plays a central role in all  further calculations. In this Section, we define it for the gain part of the collision operator.

We integrate the gain operator \eqref{collision operator gain} over the velocity space  against some suitable test function $\psi(v)$. After performing the change of variables  $(v,v_*,\sigma)\leftrightarrow (v',v'_*,\sigma')$, with $v, v'$ given in \eqref{collision transformation} and $\sigma'= \left|u\right|u$, we obtain the following weak form of the pairwise collision operator $Q_{ij}$ that corresponds to an interaction of species $\mathcal{A}_i$ with the species $\mathcal{A}_j$
\begin{equation}\label{weak form Qij}
\begin{split}
\int_{ \mathbb{R}^N} Q^{+}_{ij}(f,g)(v) \, \psi(v) \, \mathrm{d}v&= \int_{\mathbb{R}^N} \int_{\mathbb{R}^N} \int_{\mathbb{S}^{N-1}} f(v) \, g(v_*) \, \mathcal{B}_{ij}(\left|u\right|,\hat{u}\cdot \sigma) \, \psi(v') \, \mathrm{d} \sigma \, \mathrm{d}v_*\mathrm{d}v.
\end{split}
\end{equation}

Moreover, by changing $(v,v_*,\sigma)\leftrightarrow (v_*,v,-\sigma)$ in the last integral we get another representation
\begin{multline}\label{weak collision operator gain}
\int_{ \mathbb{R}^N} Q^{+}_{ij}(f,g)(v) \, \psi(v) \, \mathrm{d}v\\=  \int_{\mathbb{R}^N} \int_{\mathbb{R}^N} \int_{\mathbb{S}^{N-1}} f(v_*) \, g(v) \, \mathcal{B}_{ij}(\left|u\right|,\hat{u}\cdot \sigma) \, \psi(w'_*) \, \mathrm{d} \sigma \, \mathrm{d}v_*\mathrm{d}v,
\end{multline}
where $w'_*$ is
$$
w'_* = v_* - (1-r_{ij}) (\left|u\right| \sigma - u ),
$$
as introduced in \eqref{collision transformation ji}.

\textit{A different weak form representation for the gain term, associated to binary interactions for mixture via angular integration.} Following the representation introduced in \cite{IG-Alonso-Carneiro}, we need to write the integration on the sphere of any test function for a binary interaction of two particles with different masses as follows . Let $\varphi$ and $\chi$ be bounded and continuous functions. We  define 
the collision weight angular integral operator acting on the test functions $\varphi$ and $\chi$,
\begin{equation}\label{Pij}
\mathscr{P}_{ij}(\varphi, \chi)(u) := \int_{\mathbb{S}^{N-1}} \varphi(u^-_{ij}) \, \chi(u^+_{ij}) \, b_{ij}(\hat{u} \cdot \sigma)\, \md \sigma,
\end{equation}
where $u^+$ and $u^-$ are defined by
$$u^-_{ij}:=(1-r_{ij})(u-|u|\sigma)  \ \  \text{and} \ \  u^+_{ij}:= u-u^-_{ij}=r_{ij}u+(1-r_{ij})|u|\sigma.$$
Moreover, let $\tau$ and $\mathcal{R}$ denote the translation and reflection operators,
$$\tau_v\psi(x) := \psi(x-v) \ \ \text{and} \ \ \mathcal{R}\psi(x):=\psi(-x), \ v, x \in \mathbb{R}^N.$$
Then the weak formulation \eqref{weak form Qij} of the gain part of the collision operator for the cross section \eqref{cross section}
can be rewritten as
\begin{multline}\label{Pij and Qij+}
\int_{\mathbb{R}^N} Q^+_{ij}(f,g)(v)\psi(v) \md v = \int_{\mathbb{R}^N} \int_{\mathbb{R}^N}  f(v) g(v-u) |u|^{\gamma_{ij}} \mathscr{P}_{ij} (\tau_{-v}(\mathcal{R}\psi),1)(u)\md u \md v\\
 = \int_{\mathbb{R}^N} \int_{\mathbb{R}^N}  f(v-u) g(v) |u|^{\gamma_{ij}}  \mathscr{P}_{ij} (1, \tau_{-v}(\mathcal{R}\psi))(u) \md u \md v.
\end{multline}


\subsubsection{Carleman representation}

 The Carleman integral representation in the study of the Boltzmann equation for elastic binary interactions \cite[Appendix]{GambaPanfVil09} plays an important role in the analysis as much as in approximation theory to the studies of the initial value problem associated to this model. It provides a strong form alternate representation of the gain operator where the angular integration is performed in the orthogonal direction to the one corresponding to the difference of a molecular velocity $v$ and its post collisional one $v'$. 
 Such framework  has been useful,  both, in the study of propagation of $L^\infty$-estimates of classical solutions for elastic interactions with hard potentials and integrable cross sections \cite{GambaPanfVil09}  or for hard potentials with non-integrable cross sections in \cite{IG-P-T} . It was also used in the gain of integrability properties for the elastic Boltzmann equation for hard potentials \cite{Alonso-Carneiro, IG-Alonso} to obtain explicit estimates in $L^p$, that are sharp in some cases.

In the case of a system of Boltzmann equations for binary interactions for mixture of gases, the analogue to the Carleman representation, has only been recently addressed in \cite{BriantDaus16}  where some constant parameters are undetermined. 

We present here a new formulation  for the Carleman representation  for the  strong collisional form associated to the gain operator for a binary interaction, that is  compatible with the Banach spaces and $L^1_k$ norms  that allowed us to construct vector valued solution in \cite{IG-P-C}. All parameters in our new representation  are determined by functions of the corresponding two different masses  in the interaction.

\begin{theorem}[Carleman representation of the gain term] \label{Th Carleman}
	Let $$\p= \frac{\left( v - (2r_{ij}-1) v' \right)}{2(1-r_{ij})}, \quad r_{ij}\in (0,1).$$
	Denote with $\E$  the hyperplane  orthogonal to the vector $v-v'$, 
	that is
	\begin{equation}\label{hyperplane E}
	\E= \left\{  y\in \mathbb{R}^N: (v-v') \cdot y =0   \right\} \subset \mathbb{R}^{N-1}.
	\end{equation}
	Let $f$ and $g$ be nonnegative functions. 
	Then the gain term \eqref{collision operator gain} can be represented as follows
	\begin{multline}\label{carleman representation}
	Q^{+}_{ij}(f,g)(v) = (1-r_{ij})^{-N+1} \int_{x \in \mathbb{R}^N} \frac{f(x)}{\left|x-v\right|} \int_{ z \in E_{v x}}  g(z+P_{r_{ij}}(v,x) ) \\ \times \left| \frac{(v-x)}{2(1-r_{ij})}+z\right|^{2-N}  \mathcal{B}_{ij}\left( \left|\frac{(v-x)}{2(1-r_{ij})}+ z\right|, 1- \frac{ \left|x-v\right|^2}{2(1-r_{ij})^2 \left|\frac{(v-x)}{2(1-r_{ij})}+ z\right|^2}\right) \mathrm{d}z \, \mathrm{d}x. 
	\end{multline}
\end{theorem}
The proof of this Theorem is given in Appendix \ref{Appendix proof Carleman}.

\subsubsection{Kernel form}
By means of the Carleman representation \eqref{carleman representation}, the 
operator $\QFGp$ can be written in a  kernel form as follows.
\begin{lemma}(Kernel form of the gain operator) Let $\F = \left[f_i\right]_{1\leq i \leq I}$ and $\G = \left[g_i\right]_{1\leq i \leq I}$, where $f_i(v)\geq 0$ and $g_i(v)\geq0$ for all $v \in \mathbb{R}^N$ and all $1\leq i \leq I$.
	Then the gain operator  $\QFGp$  \eqref{gain term vector} can be written in the following kernel form  
	\begin{equation*}
	\left[\QFGp\right]_i(v) = \int_{\mathbb{R}^N} f_i(x)  \, \K \, \mathrm{d}x,
	\end{equation*}
	where the kernel is
	\begin{equation*}
	\K= \sum_{j=1}^I \tau_x Q_{ij}^+(\delta_0, \tau_{-x}g_j)(v),
	\end{equation*}
	with the translation operator $\tau_w$ defined by
	\begin{equation*}
	\tau_w g(v) = g(v-w),\quad \text{for any} \quad v, w \in \mathbb{R}^N.
	\end{equation*}
\end{lemma}
\begin{proof}
	The proof of this lemma uses the Carleman representation for the gain operator for a binary interaction in species mixture  \eqref{carleman representation}, proven in Appendix A. Indeed, it follows
	\begin{multline*}
	Q^+_{ij}(\delta_0,g)(v)=\left(1-r_{ij}\right)^{-N+1} \frac{1}{\left|v\right|} \int_{z\in \left\{ v \right\}^\perp} g\left(\frac{v}{2(1-r_{ij})} + z\right) \left| \frac{v}{2(1-r_{ij})}  + z\right|^{2-N} \\ \times \mathcal{B}_{ij}\left(\left| \frac{v}{2(1-r_{ij})} + z\right|, 1-\frac{\left|v\right|^2}{2(1-r_{ij})^2\left|\frac{v}{2(1-r_{ij})}  + z\right|^2}\right) \mathrm{d}z.
	\end{multline*}
	Then the kernel representation is obtain by using the translation operator.
\end{proof}

\subsection{Propagation of polynomial and exponential weighted $L^p$ norms}
We start by recalling that in \cite{IG-P-C} the authors proved existence and uniqueness of the vector value solution of the Cauchy problem for the homogeneous Boltzmann system of equations and generation and propagation of polynomial and exponential weighted $L^1$ norms by means of an existence theorem for ODE systems in suitable Banach spaces. We stress that norms are generated in $L^1_k, \, k \geq 0$ (observables) because of the use of Jensen's inequality for probability. In the present manuscript, we cannot have an analogue of such an inequality for $L^p, \, p>1$; therefore, we will prove the propagation of polynomial and exponential weighted $L^p$ norms. In order to do so, we need to obtain a lower bound for the negative contribution of the loss term and upper bounds for the gain part of the collision operator that produces a gain of integrability, by meaning that $L^p_k$ norm of the positive part of the collision operator is controlled sublinearly by the $L^p_k$ norms of the input functions.

\begin{theorem}\label{Gen and prop Lp norms}(Propagation of polynomially weighted $L^p$ norms.) If $\F$ is a solution of the Boltzmann system
\begin{equation}\label{Cauchy problem}
\left\{ \begin{split} & \partial_t \F(t,v) = \Q(t,v), \quad  t>0, \ v \in \mathbb{R}^N, \\& \F(0,v) = \F_0(v),\end{split} \right.
\end{equation} with the cross section \eqref{cross section} where $b_{ij} \in L^1(\mathbb{S}^{N-1})$, an initial data $\F_0:=\F(0,\cdot) \, \in \Omega$, with $\Omega$ defined in \eqref{Omega}, and $\|\F_0\|_{L^p_k}^p < \infty$  for  $k>k^*$ , with $k^*$ as given in \eqref{k star}, then there is a constant $D_{k}=\left( \frac{B_k}{A_k}\right)^{-\frac{p}{1-\theta}}$, where $B_k$ and $A_k$ are defined as
\begin{equation}\label{const Q}
\begin{split}
A_{k} &=  \frac{\min_{1\leq i,j, \leq I} \|b_{ij}\|_{L^1(\mathbb{S}^{N-1})}}{\max_{1\leq j \leq I} m_j} c_{l b}  -   \epsilon^{\bar{\gamma}}  I^{2-\frac{1}{p}} 2^{\frac{p-1}{p}} \|\F\|_{\ljedk} \Cijroner \\
&\qquad- \left( 2 \frac{\sum_{k=1}^I m_k}{ \min_{1\leq i \leq I}{m_i}} \right)^{\bar{\gamma}}  \mathfrak{C} \epsilon^{\bar{\gamma}} \| \F \|_{L^1_{k + \bar{\gamma}\left( 1 + \frac{1}{p}\right)}} , \\[0.2cm]
B_k &=  I^{2-\frac{1}{p}} 2^{\frac{p-1}{p}} \epsilon^{2+\underline{\gamma} - N}  \hat{C}_N \|\F\|_{\ljedk} \| \F\|^{2-\theta}_{L^1_{\frac{N-2}{1-\theta}+k}},
\end{split}
\end{equation} such that
\begin{equation}\label{propagation of Lp}
\|\F\|^p_{L^p_k} \leq \max\{D_k, \|\F_0\|_{L^p_k}^p \}.
\end{equation}

		\end{theorem}

\

We can also state  the theorem for propagation of $L^p$ norms with exponential weights, for $p \in (1, \infty)$. The idea behind the proof will follow from what is proven for polynomial weights.

\begin{theorem}(Propagation of exponentially weighted $L^p$ norms.) \label{theorem exp Lp}
Let $\F$ be the solution of the Boltzmann system \eqref{Cauchy problem} with the cross section \eqref{cross section} where $b_{ij} \in L^1(\mathbb{S}^{N-1})$ and let
\begin{equation}\label{sup bar gamma}
\bar{\gamma} = \max_{1\leq i,j \leq I} \gamma_{ij} \in (0,1].
\end{equation} Assume that the initial data $\F_0:=\F(0,\cdot)$ satisfies the assumptions of Theorem \ref{Gen and prop Lp norms}. If additionally \begin{equation*}
\expLoneLpFzero = C_0^e < \infty,
\end{equation*} 
for some $s \in (0,1]$, $p \in (1,\infty)$ and positive constants $\alpha_0$ and $C_0^e$, then, there exist a positive constant $\alpha$ such that
\begin{equation}\label{prop exp weight}
 \expLpF \leq \max \left\{\left( \frac{\hat{B_0}}{\hat{A_0}}\right)^{-\frac{1}{1-\theta}}, \|\F_0 e^{\alpha \langle \cdot \rangle^s}\|_{\lpnula}\right\}, \quad t \geq 0.
\end{equation}
with $\hat{B_0}$ and $\hat{A_0}$ given as
\begin{equation}\label{const Bzero Azero}
\begin{split}
\hat{A_{0}} &= \min_{1\leq i,j \leq I}  \bbjed \frac{1}{\max_{1 \leq j \leq I} m_j} c_{l b}  -   \epsilon^{\bar{\gamma}}  I^{2-\frac{1}{p}} 2^{\frac{p-1}{p}} \|\F e^{\alpha \langle \cdot \rangle^s}\|_{\lnula} \Cijroner, \\
\hat{B_0} &=  I^{2-\frac{1}{p}} 2^{\frac{p-1}{p}} \epsilon^{2+\underline{\gamma} - N}  \hat{C}_N \| \F e^{\alpha \langle \cdot \rangle^s} \|^{2-\theta}_{L^1_{\frac{N-2}{1-\theta}}},
\end{split}
\end{equation} 
and $\theta$ as given in \eqref{theta r N}.
\end{theorem}

Moreover, we can extend this results for the case $p=\infty$.
\begin{theorem}(Propagation of polynomially weighted $L^\infty$ norms.)\label{propagation of l inf norm pol weight}
Let $\gamma_{ij} \in (0,1]$, $b_{ij} \in L^1(\mathbb{S}^{N-1})$, and an initial data satisfying the hypothesis of Theorem \ref{Gen and prop Lp norms} and such that
\begin{equation*}
\| \F_0 \|_{L^{\infty}_k} = \mathcal{C}_0,
\end{equation*}  for  $k>k^*$ with $k^*$ as given in \eqref{k star}, and for some positive constant $\mathcal{C}_0$. Then there exists a constant $\mathcal{C}(\F_0)$ depending on $\bar{\gamma}, \, m_i, \, b_{ij}, k$ such that
\begin{equation}\label{ineq l inf norm pol weight}
\| \F(t, \cdot) \|_{L^{\infty}_k} \leq \mathcal{C}(\F_0), \quad t\geq 0,
\end{equation}
for $\F$ the solution of the Boltzmann system \eqref{Cauchy problem}, with $\bar{\gamma}$ defined as in \eqref{sup bar gamma}.
\end{theorem}

\begin{theorem}(Propagation of exponentially weighted $L^\infty$ norms.) \label{thm prop exp weight inf norm}
Let $\F$ be the solution of the Boltzmann system \eqref{Cauchy problem} with the cross section \eqref{cross section} where $b_{ij} \in L^1(\mathbb{S}^{N-1})$ and let $\bar{\gamma}$ defined as in \eqref{gamma bar}. Assume that the initial data $\F_0:=\F(0,\cdot)$ satisfies the assumptions of Theorem \ref{Gen and prop Lp norms} and that
\begin{equation*}
\| \F_0 \|_{\linfnula} = \mathcal{C}_0.
\end{equation*} 
If additionally, \begin{equation*}
\expLoneLpFzero = C_0^e < \infty,
\end{equation*} 
for some $s \in (0,1]$, $p \in (1,\infty)$ and positive constants $\alpha_0$ and $C_0^e$, then, there exist positive constants  $\alpha$  such that
\begin{equation}\label{prop exp weight inf norm}
 \expLinfF  \leq \max \left\{\left( \frac{\hat{B_0}}{\hat{A_0}}\right)^{-\frac{1}{1-\theta}}, \|\F_0 e^{\alpha \langle \cdot \rangle^s}\|_{\linfnula}\right\}, \quad t \geq 0.
\end{equation}
with $\hat{B_0}$ and $\hat{A_0}$ as given in \eqref{const Bzero Azero}.

\end{theorem}

\section{Estimate of the gain operator in $L^2$ framework} \label{Section Ltwo}
In this section we provide $L^2$ estimates for the gain term of the collision operator $\QFGp$ assuming that the angular part $b_{ij}(\hat{u} \cdot \sigma) \in L^\infty(\mathbb{S}^{N-1})$. First we work  in a space without weight, and then we add polynomial weight, following the strategy developed in \cite{IG-Alonso} for the single species case.
\subsection{Estimate of the gain operator in a plain $L^2$ space} As done in \cite{IG-Alonso} for the single specie case, we will start by stating and proving the $\ldvanula$ estimate for the gain part of the collision operator.
\begin{proposition}\label{proposition Q+ in L20}  Let $N\geq 3$ and let $\F$ and $\G$ be distribution functions, such that $\F \in \lnula$ and $\G \in  L^1_\frac{N-3}{1-\theta} \cap \ldvanula$. We will consider the transition probability terms  \eqref{cross section} with the angular part satisfying aditionally that
\begin{equation}\label{binLinf}
b_{ij}(\hat{u} \cdot \sigma) \in L^\infty(\mathbb{S}^{N-1}) \, \, \text{for any } \, 1\leq i,j \leq I.
\end{equation} Then the following estimate holds
	\begin{multline}\label{estimate for Ltwozero}
\left\| \QFGp\right\|_{\ldvanula} \leq \left( \sqrt{2 I \ko} \, \epsilon^{\bar{\gamma}}  \left\|\G\right\|_{\ldvanula} +  \sqrt{2 I \kN} \, \varepsilon^{2-N+\underline{\gamma}}  \left\|\G\right\|_{L^1_\frac{N-3}{1-\theta}}^{1-\theta} \left\|\G\right\|_{\ldvanula}^\theta \right) \left\|\F\right\|_{\lnula},
	\end{multline}
with $\theta=1/N$, $\varepsilon>0$, $\bar{\gamma}$ as in \eqref{sup bar gamma}, and
\begin{equation}\label{gamma bar}
\underline{\gamma}=\min_{1 \leq i, j \leq I}  \gamma_{ij},
\end{equation}
and the constants $\ko$ and $\kN$ are from Lemmas \ref{lemma estimate Q 0 in Ldva 0} and \ref{lemma estimate Q N-2 in Ldva 0} respectively. 
\end{proposition}
\begin{proof}
Once the operator $\QFGp$ is written in a kernel form, from  Minkowski's integral inequality we get an estimate 
\begin{multline*}
\left\| \QFGp \right\|_{\ldvanula} \leq \sum_{i=1}^I \left( \int_{ v \in \mathbb{R}^N} \left(   \int_{ x \in \mathbb{R}^N} f_i(x) \, \K \,  \mathrm{d}x \right)^2 \mathrm{d}v \right)^{1/2} 
\\
\leq \sum_{i=1}^I \int_{ x \in \mathbb{R}^N} f_i(x)    \left( \int_{ v \in \mathbb{R}^N}  \, \K^2 \, \mathrm{d}v   \right)^{1/2} \mathrm{d}x.
\end{multline*}
The following step is to estimate
\begin{equation}\label{pomocna 3}
 \int_{ v \in \mathbb{R}^N}  \, \K^2 \, \mathrm{d}v \quad \text{for the specific choice of the cross section \eqref{cross section}.}
\end{equation}
Whenever explicitly needed, we will use an extra subindex to clarify the cross section we use, i.e. 
\begin{equation*}
\Kg := \K, \  \text{and} \ \Qpg(f,g)(v):=Q^+_{ij}(f,g)(v)  \ \text{with} \ \mathcal{B}_{ij} \text{\ from  \eqref{cross section}. } 
\end{equation*}
Note that notation $\Kg$ is ambiguous, since the kernel $K_i$ does not depend on $j$.
Then, since
$$
\left| u \right|^{\gamma_{ij}} \leq \varepsilon^{\gamma_{ij}} \mathds{1}_{\left| u \right|\leq \varepsilon} + \left| u \right|^{\gamma_{ij}}  \mathds{1}_{\left| u \right| > \varepsilon},
$$
\eqref{pomocna 3} becomes
\begin{multline*}
 \int_{ v \in \mathbb{R}^N}  \, \Kg^2 \, \mathrm{d}v  \leq 2  \varepsilon^{2{\bar{\gamma}}}   \int_{ v \in \mathbb{R}^N}  \, \Kn^2 \, \mathrm{d}v \\ + 2 \varepsilon^{2(2-N+{\underline{\gamma}})} \int_{ v \in \mathbb{R}^N}  \, \KNmdva^2 \, \mathrm{d}v
 \\
 \leq 2 I \varepsilon^{2{\bar{\gamma}}}   \sum_{j=1}^I \int_{ v \in \mathbb{R}^N} \left(\tau_x\Qpn(\delta_0,\tau_{-x}g_j)(v)\right)^2 \mathrm{d}v \\ + 2 I \varepsilon^{2(2-N+{\underline{\gamma}})}\sum_{j=1}^I \int_{ v \in \mathbb{R}^N} \left(\tau_x \QpNmdva(\delta_0,\tau_{-x}g_j)(v)\right)^2\mathrm{d}v,
\end{multline*}
since $2-N+\gamma_{ij}\leq 0$ for $N\geq 3$. 
The final estimate \eqref{estimate for Ltwozero} follows from the following two Lemmas.
\end{proof}
\begin{lemma}\label{lemma estimate Q N-2 in Ldva 0} Let $N\geq 3$ and denote $\G=\left[g_j\right]_{1\leq j \leq I}$, with $g_j(v)\geq 0$ for all $v \in \mathbb{R}^N$ and all $1 \leq j \leq I$. Assume that the cross section takes the form \eqref{cross section} with the angular part satisfying the extra assumption \eqref{binLinf}.
 Then the following estimate holds
\begin{equation}\label{estimate Q N-2}
\sum_{j=1}^I \int_{ v \in \mathbb{R}^N} \QpNmdva(\delta_0, g_j)(v)^2 \mathrm{d}v
\\  
\leq \kN \left\|  \G  \right\|_{L^1_\frac{N-3}{1-\theta}}^{2(1-\theta)}  \left\| \G  \right\|_{\ldvanula}^{2 \theta},
\end{equation}
with $\theta=\frac{1}{N}$,
\begin{equation*}
\kN = I^{1-\theta}\tilde{C}_N   2^{N-2} \left|\mathbb{S}^{N-2}\right|  \max_{1 \leq i,j \leq I} \left( \bb^2
\left(1-r_{ij}\right)^{-N}  \left( \sqrt{\frac{\sum_{i=1}^{I}m_i}{m_j}} \right)^{\!\!\!N-3} \right),
\end{equation*}

\begin{equation*}
\tilde{C}_N = \pi^{1/2} \frac{\Gamma(N/2-1/2)}{\Gamma(N-1/2)} \left(\frac{\Gamma(N/2)}{\Gamma(N)}\right)^{-1+1/N},
\end{equation*}
 and $r_{ij}$ as in \eqref{rij}.
\end{lemma}
\begin{proof}
Since the angular part $b_{ij}$ of the cross section is assumed bounded, we can write
\begin{multline*}
 \int_{ v \in \mathbb{R}^N}  \QpNmdva(\delta_0,g_j)(v) ^2 \mathrm{d}v \\ \leq \bb^2
 \left(1-r_{ij}\right)^{2(1-N)} \int_{ v \in \mathbb{R}^N} \frac{1}{\left|v\right|^2}  \left( \int_{z \in \left\{ v \right\}^\perp} g_j\left(\frac{1}{2(1-r_{ij})} v + z\right)  \mathrm{d}z \right)^2   \mathrm{d}v.
\end{multline*}
For $v \in \mathbb{R}^N$ we pass to its spherical coordinates $r \omega$, with $r=\left|v\right| \in \mathbb{R}$ and $\omega \in \mathbb{S}^{N-1}$, and then change $r\mapsto s = \frac{1}{2(1-r_{ij})} r $, so that  the integral becomes
\begin{multline*}
\int_{ v \in \mathbb{R}^N}  \QpNmdva(\delta_0,g_j)(v) ^2 \mathrm{d}v \\ \leq \bb^2  2^{N-2} 
\left(1-r_{ij}\right)^{-N} \int_{ \omega \in \mathbb{S}^{N-1}} \int_{s\in \mathbb{R}} s^{N-3}\left( \int_{z_1\in \left\{ \omega \right\}^\perp} g_j\left(s \omega + z_1\right)  \mathrm{d}z_1 \right) \\ \times \left( \int_{z_2\in \left\{ \omega \right\}^\perp} g_j\left(s\omega + z_2\right)  \mathrm{d}z_2 \right) \mathrm{d}s \, \mathrm{d}\omega.
\end{multline*}
For fixed $\omega$, we combine integration with respect to $z_1 \in \left\{\omega \right\}^\perp $ and in the direction of $\omega$ with magnitude $s $ to form an integration in $\mathbb{R}^N$ with respect to the new variable $y= z_1 + s \omega$. Calculating $y\cdot\omega=s$, we have
\begin{multline*}
\int_{ v \in \mathbb{R}^N}  \QpNmdva(\delta_0,g_j)(v) ^2 \mathrm{d}v \\ \leq \bb^2  2^{N-2} 
\left(1-r_{ij}\right)^{-N} \int_{ \omega \in \mathbb{S}^{N-1}} \int_{y\in \mathbb{R}^N} (y \cdot \omega)^{N-3} g_j(y) \\  \times \left( \int_{z_2\in \left\{ \omega \right\}^\perp} g_j\left( (y \cdot \omega)\omega + z_2\right)  \mathrm{d}z_2 \right) \mathrm{d}y \, \mathrm{d}\omega.
\end{multline*}
Moreover, in the last integral, we change $z_2 \mapsto z=z_2 + (y\cdot \omega)\omega - y$, and noting that $z$ still belongs to the same space as $z_2$ because $z\cdot \omega =0$, we have 
\begin{multline*}
\int_{ v \in \mathbb{R}^N}  \QpNmdva(\delta_0,g_j)(v) ^2 \mathrm{d}v \\ \leq \bb^2  2^{N-2} 
\left(1-r_{ij}\right)^{-N} \int_{ \omega \in \mathbb{S}^{N-1}} \int_{y\in \mathbb{R}^N} (y \cdot \omega)^{N-3} g_j(y) \\  \times \left( \int_{z\in \left\{ \omega \right\}^\perp} g_j( y+z)  \mathrm{d}z \right) \mathrm{d}y \mathrm{d}\omega.
\end{multline*}
Then bounding the term $y\cdot \omega \leq \left|y\right|$, for $N\geq 3$, and using the representation with the Dirac delta function, we have
\begin{multline*}
\int_{ v \in \mathbb{R}^N}  \QpNmdva(\delta_0,g_j)(v) ^2 \mathrm{d}v  \leq \bb^2  2^{N-2} 
\left(1-r_{ij}\right)^{-N} \\ \times \int_{ \omega \in \mathbb{S}^{N-1}} \int_{y\in \mathbb{R}^N}  \int_{z\in \mathbb{R}^N}  \left|y\right|^{N-3} g_j(y) \, g_j( y+z) \,  \delta_0( \omega \cdot z) \, \mathrm{d}z \, \mathrm{d}y \, \mathrm{d}\omega.
\end{multline*}
Calculating,
\begin{equation*}
 \int_{ \omega \in \mathbb{S}^{N-1}}  \delta_0( \omega \cdot z) \mathrm{d}\omega = \frac{\left|\mathbb{S}^{N-2}\right|}{\left|z\right|},
\end{equation*}
we  get
\begin{multline*}
\int_{ v \in \mathbb{R}^N}  \QpNmdva(\delta_0,g_j)(v) ^2 \mathrm{d}v \\ \leq \bb^2  2^{N-2} 
\left(1-r_{ij}\right)^{-N}  \left|\mathbb{S}^{N-2}\right|   \int_{y\in \mathbb{R}^N}  \int_{z\in \mathbb{R}^N}  \left|y\right|^{N-3} \frac{g_j(y) \, g_j( y+z)}{ \left|z\right|}\mathrm{d}z \mathrm{d}y.
\end{multline*}
We can bound $\left|y\right|$ in terms of its $j-$th Lebesgue weight,
\begin{equation}\label{abs value in terms of jap brac}
\left|y\right| \leq \japyj \left( \sqrt{\frac{\sum_{i=1}^{I}m_i}{m_j}} \right),
\end{equation}
and then denote
\begin{equation*}
\tilde{g}_j(y)  = g_j(y) \japyj^{N-3},
\end{equation*}
so that the integral becomes, after translation in $z$ variable  
\begin{equation*}
\int_{ v \in \mathbb{R}^N}  \QpNmdva(\delta_0,g_j)(v) ^2 \mathrm{d}v \leq 
 C_{N, ij} \int_{y\in \mathbb{R}^N}  \int_{z\in \mathbb{R}^N}  \frac{\tilde{g}_j(y) \, g_j(z)}{\left|z-y\right|} \mathrm{d}z \mathrm{d}y,
 \end{equation*}
with
\begin{equation}\label{CNij}
C_{N, ij} =   \bb^2  2^{N-2} \left|\mathbb{S}^{N-2}\right| 
\left(1-r_{ij}\right)^{-N}\left( \sqrt{\frac{\sum_{i=1}^{I}m_i}{m_j}} \right)^{N-3}.
\end{equation}
Applying Hardy-Littlewood-Sobolev inequality, we obtain
\begin{multline*}
 \int_{y\in \mathbb{R}^N}  \int_{z\in \mathbb{R}^N}  \frac{\tilde{g}_j(y) \, g_j(z)}{\left|z-y\right|} \mathrm{d}z \mathrm{d}y 
 \\
 \leq \tilde{C}_N \left(  \int_{y\in \mathbb{R}^N}  \tilde{g}_j(y)^{\frac{2N}{2N-1}}  \mathrm{d}y  \int_{z\in \mathbb{R}^N}   g_j(z)^{\frac{2N}{2N-1}} \mathrm{d}z\right)^{\frac{2N-1}{2N}}
 \\
 \leq \tilde{C}_N \left(  \int_{y\in \mathbb{R}^N}  \tilde{g}_j(y)^{\frac{2N}{2N-1}}  \mathrm{d}y  \right)^{\frac{2N-1}{N}},
\end{multline*}
with the exact constant, as proved in \cite[Theorem 4.3]{L-L}, given by
\begin{equation}\label{CtildaN}
\tilde{C}_N = \pi^{1/2} \frac{\Gamma(N/2-1/2)}{\Gamma(N-1/2)} \left(\frac{\Gamma(N/2)}{\Gamma(N)}\right)^{-1+1/N}.
\end{equation}
Then using log-convexity of $L^p$ norms, 
\begin{multline*}
 \int_{y\in \mathbb{R}^N} \left( {g}_j(y)\japvj^{N-3} \right)^{\frac{2N}{2N-1}}  \mathrm{d}y 
\\ \leq \left( \int_{y\in \mathbb{R}^N}  {g}_j(y)\japvj^\frac{N-3}{1-\theta}   \mathrm{d}y \right)^{\frac{2N(1-\theta)}{2N-1}}  \left( \int_{x\in \mathbb{R}^N}  {g}_j(x)^2   \mathrm{d}x \right)^{\frac{N \theta}{2N-1}}, \quad \theta:= \frac{1}{N}. 
\end{multline*}
Therefore,
\begin{multline*}
\int_{ v \in \mathbb{R}^N}  \QpNmdva(\delta_0,g_j)(v) ^2 \mathrm{d}v \\\leq 
C_{N,ij} \tilde{C}_N \left( \int_{y\in \mathbb{R}^N}  {g}_j(y)\japvj^\frac{N-3}{1-\theta}   \mathrm{d}y \right)^{2(1-\theta)}  \left( \int_{x\in \mathbb{R}^N}  {g}_j(x)^2   \mathrm{d}x \right)^{\theta}, \quad \theta= \frac{1}{N}.
\end{multline*}
In order to represent this estimate in norm notation, we sum the last inequality over $j=1, \dots, I$, obtaining
\begin{multline*}
\sum_{j=1}^I \int_{ v \in \mathbb{R}^N} \QpNmdva(\delta_0, g_j)(v)^2 \mathrm{d}v  
\leq  I^{1-\theta} \tilde{C}_N \max_{1 \leq i,j \leq I} (C_{N, ij}) \\
 \times \left( \sum_{j=1}^I \int_{y\in \mathbb{R}^N}  {g}_j(y)\japvj^\frac{N-3}{1-\theta}   \mathrm{d}y \right)^{2(1-\theta)} \left(  \sum_{j=1}^I \int_{x\in \mathbb{R}^N}  {g}_j(x)^2   \mathrm{d}x \right)^{\theta},
\end{multline*}
since $2(1-\theta) \geq 1 $ and $\theta<1$ for $\theta= \frac{1}{N}$ and $N\geq2$.  Using the norm notation \eqref{norm definition}, we finally get estimate \eqref{estimate Q N-2}.
\end{proof}

\begin{lemma}\label{lemma estimate Q 0 in Ldva 0} Assume $N\geq 3$ and the cross section in the form \eqref{cross section} with the angular part satisfying \eqref{binLinf}. Let $\G=\left[g_j\right]_{1\leq j \leq I}$, with $g_j(v)\geq 0$ for all $v \in \mathbb{R}^N$ and all $1 \leq j \leq I$. Then the following estimate holds
	\begin{equation*}
	\sum_{j=1}^I \int_{ \mathbb{R}^N} \Qpn(\delta_0, g_j)(v)^2 \mathrm{d}v\leq \ko \left\| \G \right\|_{\ldvanula},
	\end{equation*}
	with a constant  $\ko = \max_{1 \leq i,j \leq I} \left(  \bb^2 \left(C_N^{ij}\right)^2 \right)$ and
	\begin{equation*}
	C_N^{ij}=\left|\mathbb{S}^{N-2}\right| \int_{-1}^1 \left(\sqrt{2} (1-r_{ij})( 1+ \mu )\right)^{-N/2} \left(1-\mu^2\right)^{\frac{N-3}{2}} \mathrm{d}\mu.
	\end{equation*}
\end{lemma}
\begin{proof}
The proof starts with  the weak form of the gain term \eqref{weak collision operator gain}, after a rotation  $\sigma \mapsto - \sigma$,
\begin{multline*}
\int_{ \mathbb{R}^N} \Qpn(\delta_0,g_j)(v) \, \psi(v) \, \mathrm{d}v\\ \leq \bb  \int_{\mathbb{R}^N}\int_{\mathbb{R}^N}\int_{ \mathbb{S}^{N-1}} \delta_0(v_*) \, g_j(v) \,  \psi(v_*- (1-r_{ij}) (\left|u\right| \sigma - u )) \, \mathrm{d} \sigma  \, \mathrm{d}v  \, \mathrm{d}u
\\
= \bb  \int_{\mathbb{R}^N} g_j(v)  \int_{ \mathbb{S}^{N-1}}  \psi( (1-r_{ij}) (\left|v\right| \sigma + v )) \, \mathrm{d} \sigma  \, \mathrm{d}v.
\end{multline*}
Applying H\"{o}lder inequality,  
\begin{multline}\label{pomocna 1}
\int_{ \mathbb{R}^N} \Qpn(\delta_0,g_j)(v) \, \psi(v) \, \mathrm{d}v \leq \bb \left( \int_{\mathbb{R}^N} g_j(v)^2 \mathrm{d}v \right)^{1/2} \\ 
\times \left( \int_{ \mathbb{R}^N} \left( \int_{ \mathbb{S}^{N-1}}  \psi( (1-r_{ij}) (\left|v\right| \sigma + v )) \, \mathrm{d} \sigma \right)^2 \, \mathrm{d}v \right)^{1/2}.
\end{multline}
Assuming $\psi$ is a radial function, that is 
$$ \psi\left( (1-r_{ij}) (\left|v\right| \sigma + v )\right) =\psi\left( \sqrt{2} (1-r_{ij}) \left|v\right|( 1+ \sigma \cdot \hat{v} )\right),$$
 the integration with respect to $\sigma$ can be simplified
\begin{multline*}
 \int_{ \mathbb{S}^{N-1}}  \psi( (1-r_{ij}) (\left|v\right| \sigma + v )) \, \mathrm{d} \sigma \\= \left|\mathbb{S}^{N-2}\right| \int_{-1}^1 \psi\left( \sqrt{2} (1-r_{ij}) \left|v\right|( 1+ \mu )\right) \left(1-\mu^2\right)^{\frac{N-3}{2}} \mathrm{d}\mu.
\end{multline*}
Then, by the Minkowski inequality,
\begin{multline*}
 \left|\mathbb{S}^{N-2}\right|  \left( \int_{ \mathbb{R}^N} \left( \int_{-1}^1 \psi\left( \sqrt{2} (1-r_{ij}) \left|v\right|( 1+ \mu )\right) \left(1-\mu^2\right)^{\frac{N-3}{2}} \mathrm{d}\mu \right)^2 \, \mathrm{d}v \right)^{1/2}
 \\
 \leq \left|\mathbb{S}^{N-2}\right|   \int_{-1}^1  \left(  \int_{ \mathbb{R}^N}  \psi\left( \sqrt{2} (1-r_{ij}) \left|v\right|( 1+ \mu )\right)^2   \mathrm{d}v \right)^{1/2} \left(1-\mu^2\right)^{\frac{N-3}{2}} \mathrm{d}\mu
 \\
 =    C_N^{ij}  \left(  \int_{ \mathbb{R}^{N}}  \psi(\left|v\right|)^2    \mathrm{d}v \right)^{1/2}
\end{multline*}
with 
\begin{equation*}
C_N^{ij}=\left|\mathbb{S}^{N-2}\right| \int_{-1}^1 \left(\sqrt{2} (1-r_{ij})( 1+ \mu )\right)^{-N/2} \left(1-\mu^2\right)^{\frac{N-3}{2}} \mathrm{d}\mu.
\end{equation*}
Therefore, \eqref{pomocna 1} becomes
\begin{multline*}
\int_{ \mathbb{R}^N} \Qpn(\delta_0,g_j)(v) \, \psi(v) \, \mathrm{d}v\\ 
\leq \bb C_N^{ij} \left( \int_{\mathbb{R}^N} g_j(v)^2 \mathrm{d}v \right)^{1/2}  \left(  \int_{ \mathbb{R}^{N}}  \psi(\left|v\right|)^2    \mathrm{d}v \right)^{1/2}.
\end{multline*}
We finally get
\begin{multline*}
	\sum_{j=1}^I \int_{ v \in \mathbb{R}^N} \Qpn(\delta_0, g_j)(v)^2 \mathrm{d}v
	\\ \leq   \max_{1 \leq i,j \leq I} \left(  \bb^2 \left(C_N^{ij}\right)^2 \right)  \sum_{j=1}^I \int_{\mathbb{R}^N} g_j(v)^2 \mathrm{d}v \\ =  \max_{1 \leq i,j \leq I} \left(  \bb^2 \left(C_N^{ij}\right)^2 \right) \left\| \G \right\|_{\ldvanula},
\end{multline*}
which concludes the proof.
\end{proof}

Therefore, with the completion of these two lemmas \ref{lemma estimate Q N-2 in Ldva 0} and \ref{lemma estimate Q 0 in Ldva 0}, the proof of \eqref{estimate for Ltwozero} from proposition \ref{proposition Q+ in L20} is completed.

\subsection{Estimate of the gain operator in a polynomially weighted $L^2$ space}The  estimate of the gain operator $ \QFGp$ in $\ldvak$ norm makes use of the estimate in  $\ldvanula$ and the following inequality 
\begin{equation*}
\japvi \leq \japvip \japvjps,
\end{equation*}
which immediately follows from the conservation law of kinetic energy \eqref{CL micro}. Indeed, it holds
\begin{equation*}
\left\| \QFGp\right\|_{\ldvak} \leq \left\| \QFGptil\right\|_{\ldvanula},
\end{equation*}
where 
$$
[\tilde{\F}]_i(v) =  f_i(v) \japvi^k, \quad i=1,\dots,I,
$$
having in mind $\left[\F\right]_i(v) = f_i(v)$. Then we apply the result of Proposition \ref{proposition Q+ in L20} to $ \QFGptil$.
Therefore, the following Proposition holds.
\begin{proposition} \label{prop ltwok qplus} Let $N\geq 3$ and let $\F$ and $\G$ be distribution functions, such that  $\F \in \lk$ and $\G \in  L^1_{\frac{N-3}{1-\theta}+k} \cap \ldvak$. For the transition probability terms choose \eqref{cross section},  where $b_{ij} \in L^\infty(\mathbb{S}^{N-1})$. Then the following estimate holds
	\begin{multline*}
	\left\| \QFGp\right\|_{\ldvak} \leq \left( \sqrt{2 I \ko} \,  \epsilon^{\bar{\gamma}}  \left\|\G\right\|_{\ldvak} 
\right.	\\ \left.
	+  \sqrt{2 I \kN} \,  \varepsilon^{2(2-N+{\underline{\gamma}})} \left\|\G\right\|_{L^1_{\frac{N-3}{1-\theta}+k}}^{1-\theta} \left\|\G\right\|_{\ldvak}^\theta \right) \left\|\F\right\|_{\lk},
	\end{multline*}
	where the constants are identical to those in Proposition \ref{proposition Q+ in L20}.
\end{proposition}

\section{Estimate of the gain operator in $L^p$ framework, $p \in (1,\infty)$} \label{Section Lp}

 The goal on the next section is to elaborate on the weak formulation of the gain operator formulated in \eqref{Pij and Qij+} in order to obtain a Young's inequality for each gain term of the collision operator associated to a pair of interacting species, for that we need to extend the results of \cite{IG-Alonso-Carneiro} for the multispecies framework.

\subsection{Estimate of the gain operator in a plain $L^p$ space}
We first study the integrability properties of the collision weight angular integral operator $\mathscr{P}_{ij}$, since they are closely related to those of $Q^+_{ij}$, by \eqref{Pij and Qij+}.


\begin{theorem}[$L^r$ control of the  collision weight angular integral operator]\label{trm Lr norm of P}
	Let $1 \leq p,q,r \leq \infty$ with $1/p+1/q=1/r$. Then if $r\neq \infty$, for the operator $\opP$ \eqref{Pij} with $b_{ij} \in L^1(\mathbb{S}^{N-1})$, the following estimate holds,
	\begin{equation}\label{estimate on Pij}
	\| \opP  \|_{\lrcomp} \leq \CopPr \,  \| \varphi  \|_{\lpcomp}  \  \| \chi  \|_{\lqcomp},
	\end{equation}
for any $\varphi \in L^p(\mathbb{R}^N)$ and $\chi \in L^q(\mathbb{R}^N)$, where the constant is
\begin{multline}\label{Cij from Lr norm of P}
\CopPr
=|\mathbb{S}^{N-2}|\int_{-1}^1 \left((1-r_{ij})^2(2-2s)\right)^{-\frac{N}{2p}} \\ \times \left(r_{ij}(1-r_{ij}) \left( \frac{1}{r_{ij}(1-r_{ij})} -2 +2s \right)\right)^{-\frac{N}{2q}} \mdx,
\end{multline}
with the measure $\xi_N^{b_{ij}}$ defined on $[-1,1]$ as 
\begin{equation}\label{measure}
\mdx=b_{ij}(s)(1-s^2)^{\frac{N-3}{2}}\, \md s.
\end{equation}
For $r=\infty$	the estimate \eqref{estimate on Pij} still holds but the constant simplifies to 
\begin{equation}\label{Cij from L inf norm of P}
\CopPinf:=\CopPinftri=\bbjed.
\end{equation}
\end{theorem}

\begin{proof} The proof consists in multiple steps.
\\
\textbf{Step 1. (Radial symmetrization)} Following the strategy in \cite{IG-Alonso-Carneiro}, we will consider $G=SO(N)$ the group of rotations of $\mathbb{R}^N$ (orthonormal transformations of determinant 1) and we will denote $R$ a generic rotation. Moreover, we assume the Haar measure $d\mu$ of this compact topological group normalized, i.e., $\int_G d\mu(R)=1$.\\ For $p \geq 1$, let $f \in L^p(\mathbb{R}^N)$ and we will define the {radial symmetrization} $f_p^*$ of $f$ by
$$f_p^*(x) = \left( \int_G |f(Rx)|^p \md \mu(R) \right)^{\frac{1}{p}} \ , \ \text{if} \ 1 \leq p < \infty$$
and
$$f^*_{\infty} = \esssup_{|y|=|x|}\ |f(y)|$$

This radial rearrangement can be seen as an $L^p$-average of $f$ over all the rotations $R \in G$ and its properties can be found in \cite{IG-Alonso-Carneiro}.

It is not hard to prove that $$\|f\|_{\lpcomp}=\|f_p^*\|_{\lpcomp}.$$
\textbf{Step 2. (Passage from $ \mathscr{P}_{ij}$ to an one-dimensional operator $\mathscr{B}_{ij}$)} 
First we use the following lemma, whose proof only uses the structure of $\mathscr{P}_{ij}$ but not the definition of $u^-_{ij}$ and $u^+_{ij}$.
\begin{lemma}[Lemma 3 in  \cite{IG-Alonso-Carneiro}]\label{lemma Pij}
	\vspace{0.2cm}
	Let $\varphi, \chi,  \psi \in C_c(\mathbb{R}^N)$ and $1/p+1/q+1/r'=1$, with $1 \leq p, q, r' \leq \infty$. Then
	$$\left| \int_{\mathbb{R}^N} \opP(u) \psi(u) du \right| \leq \int_{\mathbb{R}^N} \opPs(u) \psi_{r'}^*(u) du.$$ 
\end{lemma}
From this Lemma, $L^p$-estimates for the operator $\mathscr{P}_{ij}$ will follow by considering radial functions. As in \cite{IG-Alonso-Carneiro} for any radial function $f:\mathbb{R}^N \rightarrow \mathbb{R}$, we can define $\tilde{f}:= \mathbb{R}^+ \rightarrow \mathbb{R}$ by
$$f(x)=\tilde{f}(|x|).$$
In addition, for any $1 \leq p < \infty$,
\begin{equation} \label{f and tilde}
\int_{\mathbb{R}^N} f(x)^p \md x = |\mathbb{S}^{N-1}| \int_0^\infty \tilde{f}(t)^p t^{N-1}\md t,
\end{equation}
and for $p=\infty$,
\begin{equation}\label{Linf R+} 
\left\| f \right\|_{\infty}= \left\| \tilde{f} \right\|_{\linfp} := \esssup_{y \in \mathbb{R}^+} \left| \tilde{f}(y) \right|.
\end{equation}
Let us show now how the operator $\mathscr{P}_{ij}$ simplifies to a 1-dimensional operator when applied to radial functions. Indeed, if $\varphi$ and $\chi$ are radial
\begin{align*}
\opP(u)&=\int_{\mathbb{S}^{N-1}} \tilde{\varphi}(|u^-_{ij}|) \, \tilde{\chi}(|u^+_{ij}|) \, b_{ij}(\hat{u}\cdot \sigma) \, \md \sigma\\
&=\int_{\mathbb{S}^{N-1}} \tilde{\varphi}(a^1_{ij}(|u|, \hat{u}\cdot \sigma) \, \tilde{\chi}(a^2_{ij}(|u|, \hat{u}\cdot \sigma)) \, b_{ij}(\hat{u}\cdot \sigma) \, \md \sigma\\
&=|\mathbb{S}^{N-2}| \int_{-1}^1 \tilde{\varphi}(a^1_{ij}(|u|, s)) \, \tilde{\chi}(a^2_{ij}(|u|, s)) \, b_{ij}(s)(1-s^2)^{\frac{N-3}{2}} \, \md s,
\end{align*}
where $a^1_{ij},a^2_{ij}:\mathbb{R}^+ \times [-1,1] \rightarrow \mathbb{R}^+$ are defined by 
\begin{equation}\label{a1 and a2}
a^1_{ij}(x,s)=(1-r_{ij})x\left(2-2s\right)^{\frac{1}{2}} \  , \   a^2_{ij}(x,s)=(r_{ij}(1-r_{ij}))^{\frac{1}{2}}x\left(\!\frac{1}{r_{ij}(1-r_{ij})}-2+2s\!\!\right)^{\!\!\frac{1}{2}}.
\end{equation}
Considering the measure $\xi_N^{b_{ij}}$ on $[-1,1]$ as $\mdx=b_{ij}(s)(1-s^2)^{\frac{N-3}{2}}\, \md s$, as given in \eqref{measure}, we conclude that
\begin{equation}\label{Pij tilde}
\widetilde{\opP}(x)= |\mathbb{S}^{N-2}|  \int_{-1}^1 \tilde{\varphi}(a_1(x, s))  \, \tilde{\chi}(a_2(x, s)) \, \mdx.
\end{equation}

Therefore, we are led to introduce the following bilinear operator for any two bounded and continuous functions $\varphi, \chi:\mathbb{R}^+ \rightarrow \mathbb{R}$ as
\begin{equation}\label{operator B}
\opB(x):=\int_{-1}^1 \varphi(a^1_{ij}(x,s)) \, \chi(a^2_{ij}(x,s)) \mdx,
\end{equation}
with mappings $a^1_{ij}$ and $a^2_{ij}$ defined in \eqref{a1 and a2}.
\\
\textbf{Step 3. (Study of the operator $\mathscr{B}_{ij}$)}
\begin{lemma} \label{lemma Bij}
	Let $1 \leq p,q,r \leq \infty$ with $1/p+1/q=1/r$, $\varphi \in L^p(\mathbb{R}^+, x^{N-1} \mathrm{d} x)$ and $\chi \in L^q(\mathbb{R}^+, x^{N-1} \mathrm{d} x)$. For $r\neq \infty$ we have
	\begin{equation}\label{estimate on B r}
	\| \opB \ | \cdot |^{\frac{N-1}{r}}\|_{\lrp} \leq \CopBr \, \| \varphi \ | \cdot |^{\frac{N-1}{p}}\|_{\lpp}  \| \chi \ | \cdot |^{\frac{N-1}{q}}\|_{\lqp},
	\end{equation}
	where the constant is 
	\begin{multline*}
	\CopBr 
	=\int_{-1}^1 \left((1-r_{ij})(2-2s)^{1/2}\right)^{-\frac{N}{p}} 
	\\ \times
	\left(r_{ij}(1-r_{ij})\left( \frac{1}{r_{ij}(1-r_{ij})} -2 +2s \right)\right)^{-\frac{N}{2q}} \mdx.
	\end{multline*}
	For $r=\infty$, we have
	\begin{equation}\label{estimate on B inf}
	\left\| \opB  \right\|_{\linfp} \leq \CopBinf  \, \left\| \varphi \right\|_{\linfp} \left\| \chi \right\|_{\linfp},
	\end{equation}
	with the constant that simplifies to 
		\begin{equation*}
		\CopBinf:= \CopBinftri 
	=\frac{1}{\left|\mathbb{S}^{N-2}\right|} \bbjed.
	\end{equation*}
\end{lemma}
\begin{proof}
Using Minkowski's and H\"older's inequalities with exponents $p/r$ and $q/r$ we obtain
\begin{multline*}
\left(\int_{0}^\infty \left| \int_{-1}^1  \varphi(a^1_{ij}(x,s)) \, \chi(a^2_{ij}(x,s)) \mdx  \right|^r x^{N-1} \mathrm{d} x \right)^{1/r}
\\ \leq \int_{-1}^1 \left( \int_0^\infty |\varphi(a^1_{ij}(x,s))|^r |\chi(a^2_{ij}(x,s))|^r x^{N-1} \mathrm{d} x  \right)^{\frac{1}{r}} \mdx
\\ \leq \int_{-1}^1 \left( \int_0^\infty |\varphi(a^1_{ij}(x,s))|^p x^{N-1} \mathrm{d} x  \right)^{\frac{1}{p}} \left( \int_0^\infty |\chi(a^2_{ij}(x,s))|^q x^{N-1} \mathrm{d} x  \right)^{\frac{1}{q}} \mdx.
\end{multline*}
Then for $\varphi$ we can consider the change of variables $y=a^1_{ij}(x,s)$, for any fixed $s \in [-1,1]$,
\begin{align*}
\left( \int_0^\infty |\varphi(a^1_{ij}(x,s))|^p x^{N-1} \mathrm{d} x  \right)^{\frac{1}{p}} = \left( (1-r_{ij})(2-2s)^{1/2}\right)^{-\frac{N}{p}} \left( \int_0^\infty |\varphi(y)|^p y^{N-1} \mathrm{d} y  \right)^{\frac{1}{p}}.
\end{align*}
And we can repeat the same for $\chi$,
\begin{multline*}
\left( \int_0^\infty |\chi(a^2_{ij}(x,s))|^q x^{N-1} \mathrm{d} x\right)^{\frac{1}{q}} \\= \left( r_{ij}(1-r_{ij})\left( \frac{1}{r_{ij}(1-r_{ij})} -2+2s \right)\right)^{-\frac{N}{2q}}\left( \int_0^\infty |\chi(y)|^q y^{N-1} \mathrm{d} y  \right)^{\frac{1}{q}},
\end{multline*}
which yields \eqref{estimate on B r}. 

For $r=\infty$, from the definition of the operator $\mathscr{B}_{ij}$ \eqref{operator B}  and pulling out the $L^\infty$ norms of $\varphi$ and $\chi$, we obtain
\begin{equation*}
\opB(x) \leq \frac{1}{\left|\mathbb{S}^{N-2}\right|} \left\| \varphi \right\|_{\linfp} \left\| \chi \right\|_{\linfp} \bbjed,
\end{equation*}
and taking the supremum we obtain \eqref{estimate on B inf}.
\end{proof}
\noindent \textbf{Step 4. (Conclusion of the proof for Theorem \ref{trm Lr norm of P})} 
From the Lemma \ref{lemma Pij}, for $1/p+1/q=1/r$ and $r\neq \infty$,  by duality we have
\begin{equation*}
\left( \int_{\mathbb{R}^N} \left|\opP\right|^r  \md u\right)^{1/r} \leq \left( \int_{\mathbb{R}^N} |\opPs|^r  \md u\right)^{1/r}
\end{equation*}
Now, using equations \eqref{f and tilde}, \eqref{Pij tilde} and Lemma \ref{lemma Bij} 
\begin{multline*}
 \left( \int_{\mathbb{R}^N} |\opPs^r \md u\right)^{1/r} 
 = \left| \mathbb{S}^{N-1} \right|^{1/r} \left( \int_{0}^\infty |\widetilde{\opPs}(x)|^r x^{N-1} \md x\right)^{1/r} \\ =
 \left| \mathbb{S}^{N-1} \right|^{1/r}  \left| \mathbb{S}^{N-2} \right| \left( \int_{0}^\infty |\opBs(x)|^r x^{N-1} \md x\right)^{1/r} \\
 \leq \CopBr \left| \mathbb{S}^{N-1} \right|^{1/r}  \left| \mathbb{S}^{N-2} \right|  \left( \int_0^\infty |\widetilde{\varphi^*_p}(x)|^p x^{N-1} \mathrm{d} x  \right)^{\frac{1}{p}}  \left( \int_0^\infty |\widetilde{\chi^*_q}(x)|^q x^{N-1} \mathrm{d} x  \right)^{\frac{1}{q}}\\
 \leq \CopBr  \left| \mathbb{S}^{N-2} \right|  \left( \int_{\mathbb{R}^N} |\varphi^*_p(u)|^p \mathrm{d} u  \right)^{\frac{1}{p}}  \left( \int_{\mathbb{R}^N} |\chi^*_q(u)|^q  \mathrm{d} u  \right)^{\frac{1}{q}}\\
 = \CopBr  \left| \mathbb{S}^{N-2} \right|  \left( \int_{\mathbb{R}^N} |\varphi(u)|^p  \mathrm{d} u  \right)^{\frac{1}{p}}  \left( \int_{\mathbb{R}^N} |\chi(u)|^q  \mathrm{d} u  \right)^{\frac{1}{q}},
\end{multline*}
which yields \eqref{estimate on Pij}. 
If $r=\infty$, by duality and Lemma \ref{lemma Pij}, we have
\begin{multline*}
\left\| \opP \right\|_{\infty} \leq \left\| \opPsinf \right\|_{\infty} =  \left\| \widetilde{\opPsinf}   \right\|_{\infty}  \\ \leq \left\| \widetilde{\varphi^*_\infty} \right\|_{\linfp}\left\| \widetilde{\chi^*_\infty} \right\|_{\linfp} \bbjed  = \bbjed \left\| \varphi \right\|_{\infty}\left\| \chi \right\|_{\infty},
\end{multline*} 
with the norm on radially symmetrized functions as defined in \eqref{Linf R+}. Now the proof of the Theorem  \ref{trm Lr norm of P} is completed.
\end{proof}

\begin{theorem}\label{trm Lr0 norm of Q+}
	Let $p, q, r \in [1,\infty]$ with $\frac{1}{p}+\frac{1}{q}=1+\frac{1}{r}$.  Assume that $\mathcal{B}_{ij}$ takes the form
\begin{equation*}
\mathcal{B}_{ij}(x,y) =x^\lij \, b_{ij}(y), \quad \lij \geq0, \quad \text{and} \quad b_{ij} \in L^1([-1,1]; \mdx),
\end{equation*}
for any $i, j = 1,\dots, I$,
with the measure from \eqref{measure}. Then  we have the following estimate for $\F \in \lplam(\mathbb{R}^N)$ and $\G \in \lqlam(\mathbb{R}^N)$
\begin{equation}\label{Q+ r norm}
\left\| \QFGp\right\|_{\lrnula} \leq \CopQr \left\|\F\right\|_{\lplam}  \left\|\G\right\|_{\lqlam},
\end{equation}
with
\begin{equation*}
\lambda = \max_{1 \leq i,j \leq I} \lij,
\end{equation*}
and 	where the constant $\CopQr$ for $r\neq 1$ and $r\neq\infty$ is given by
\begin{equation}\label{cqplus pqr}
\CopQr=I^{\frac{r-1}{r}} \max_{1 \leq i,j \leq I} \left( 2^{\lij} \Mij \CijQp \right),
\end{equation}
and
\begin{align}
\Mij &:=  \left(\frac{\sum_{k=1}^I m_k }{m_i}\right)^\frac{\lij}{2} +\left(\frac{\sum_{k=1}^I m_k}{m_j}\right)^\frac{\lij}{2},\label{Mij}\\
\CijQp&:= \left(\Cijtwo\right)^{r'/q'}  \left(\Cijthree\right)^{r'/p' } \\ &= |\mathbb{S}^{N-2}|  \left(\int_{-1}^1 \left((1-r_{ij})^2(2-2s)\right)^{-\frac{N}{2r'}} \mdx\right)^{r'/q'} \\ & \qquad \qquad \times \left(\int_{-1}^1 \left(r_{ij}(1-r_{ij})\left( {\frac{1}{r_{ij}(1-r_{ij})}} -2 +2s \right)\right)^{-\frac{N}{2r'}} \mdx \right)^{r'/p'}. \label{Cij Q+}
\end{align}
When $r=1$ and  $r=\infty$ the constants change to
\begin{equation}\label{C1 and Cinfty for Q+}
\begin{split}
\CopQone &:= \CopQonetri= \max_{1 \leq i,j \leq I} \left( 2^{\lij} \Mij \bbjed \right),\\ 
\CopQinf  &= \max_{1 \leq i,j \leq I} \left(2^{\lij} \Mij \CijQppppinf \right).
\end{split}
\end{equation}
\end{theorem}

\begin{proof}  
This proof is inspired of the one written in \cite{IG-Alonso-Carneiro}, Theorem 1. We rewrite it here in complete detail as several changes are needed for the adaptation to gas mixtures models.\\
If $\lij=0$  we can  consider the gain operator in a weak form given in \eqref{Pij and Qij+},
\begin{equation}\label{J def}
J_0:= \int_{\mathbb{R}^N} Q^+_{ij}(f_i,g_j)(v)\psi (v) \md v = \int_{\mathbb{R}^N} \int_{\mathbb{R}^N} f_i(v) g_j(v-u) \mathscr{P}_{ij}(\tau_{-v} (\mathcal{R}\psi), 1)(u) \md u \md v.
\end{equation}
When $\lij >0 $ we use the following additional inequality, 
$$|u|^\lij \leq \left(|v-u|+|v| \right)^\lij
\leq 2^\lij \left( |v-u|^\lij + |v|^\lij \right),$$ 
so that the weak form of the Gain operator for the pair $\{ij\}$ is estimated by
\begin{multline}\label{J def with weight}
J_\lij:=\int_{\mathbb{R}^N} Q^+_{ij}(f_i,g_j)  (v) \psi(v) dv \\
 = \int_{\mathbb{R}^N}\int_{\mathbb{R}^N} f_i(v)g_j(v-u) \mathscr{P}_{ij}(\tau_{-v} (\mathcal{R} \psi),1)(u) |u|^\lij \md u \md v\\
\leq \int_{\mathbb{R}^N}\int_{\mathbb{R}^N} f_i(v)g_j(v-u) \mathscr{P}_{ij}(\tau_{-v} (\mathcal{R} \psi),1)(u) ( 2^\lij(|v-u|^\lij + |v|^\lij)) \md u \md v\\
\leq 2^\lij \left( \int_{\mathbb{R}^N}\int_{\mathbb{R}^N} f_i(v)g_j(v-u)\left|v-u\right|^\lij \mathscr{P}_{ij}(\tau_{-v} (\mathcal{R} \psi),1)(u) \md u \md v \right. \\ \left. +  \int_{\mathbb{R}^N}\int_{\mathbb{R}^N} f_i(v)\left|v\right|^\lij g_j(v-u) \mathscr{P}_{ij}(\tau_{-v} (\mathcal{R} \psi),1)(u)   \md u \md v\right).
\end{multline}

The proof is separated  into three subcases, depending whether $i)$ $(p,q,r) \notin \{(1,1,1),  (p, p', \infty)\}$, $ii)$ $(p,q,r) = (1,1,1)$ or $iii)$ $(p,q,r) = (p,p',\infty)$. Each subcase contains three steps. The first step aims  at estimating $J_0$ from \eqref{J def}. Then, in the second step, by  applying the very same estimate on $J_0$, but now on functions $f_i$ and $g_j\langle\cdot\rangle_j$ (respectively on  $f_i\langle\cdot\rangle_i$ and $g_j$) we estimate $J_\lij$ given in \eqref{J def with weight}. Finally, in step three, by duality arguments, we obtain an estimate for $\|Q^+_{ij} (f_i, g_j)\|_{\lrcomp}$ that will  yield the one for $	\left\| \QFGp\right\|_{\lrnula}$.  To conclude, we make  use of the  following inequalities in order to obtain the appropriate norms on the right hand side of \eqref{Q+ r norm},
\begin{equation}\label{estimate with lij}
\|f\|_{p} \leq \|f\|_{L^p_{\lij,i}} \ \ \text{and} \ \ \|f \left|\cdot\right|^\lij \|_{p} \leq \sqrt{ \frac{\sum_{k=1}^I m_k}{m_i}} \ \|f\|_{L^p_{\lij,i}},
\end{equation}
the first one following from the monotonicity of norms, and the second one from \eqref{abs value in terms of jap brac}.
\\
\textbf{Subcase 1. ($(p,q,r) \notin \{(1,1,1),  (p, p', \infty)\}$)}.\\
\textbf{Step 1. (Estimate of $J_0$)}.
Since the exponents $p,q,r$ in the Theorem satisfy $1/p'+1/q'+1/r=1$ we can regroup the terms conveniently.
\begin{multline*}
J_0=\int_{\mathbb{R}^N} \int_{\mathbb{R}^N} \left(f_i(v)^{\frac{p}{r}} g_j(v-u)^{\frac{q}{r}}\right)  \left( f_i(v)^{\frac{p}{q'}} \mathscr{P}_{ij}(\tau_{-v} (\mathcal{R} \psi), 1)(u)^{\frac{r'}{q'}}\right)\\ \times
\left( g_j(v-u)^{\frac{q}{p'}} \mathscr{P}_{ij}(\tau_{-v} (\mathcal{R} \psi), 1)(u)^{\frac{r'}{p'}}\right)\md u \, \md v.
\end{multline*}
Then using H\"older's inequality
$$J_0\leq I_1 I_2 I_3.$$
where
\begin{align*}
I_1&:=\left(\int_{\mathbb{R}^N} \int_{\mathbb{R}^N}f_i(v)^p g_j(v-u)^q  \md u\, \md v\right)^{\frac{1}{r}}= \|f_i\|^{p/r}_{\lpcomp} \|g_j\|^{q/r}_{\lqcomp},\\
I_2&:=\left(\int_{\mathbb{R}^N} \int_{\mathbb{R}^N}f_i(v)^p \mathscr{P}_{ij}(\tau_{-v} (\mathcal{R} \psi), 1)(u)^{r'}   \md u \, \md v\right)^{\frac{1}{q'}} { \leq  \left(\Cijtwo  \right)^{r'/q'}} \|f_i\|^{p/q'}_{\lpcomp} \|\psi\|_{\lrpcomp}^{r'/q'},\\
I_3&:=\left(\int_{\mathbb{R}^N} \int_{\mathbb{R}^N}g_j(v-u)^q \mathscr{P}_{ij}(\tau_{-v} (\mathcal{R} \psi), 1)(u)^{r'}   \md u \, \md v\right)^{\frac{1}{p'}}\\
&=\left(\int_{\mathbb{R}^N} \int_{\mathbb{R}^N}g_j(v)^q \mathscr{P}_{ij}(1, \tau_{-v} (\mathcal{R} \psi))(u)^{r'}   \md u \, \md v\right)^{\frac{1}{p'}}  { \leq \left(\Cijthree\right)^{r'/p'}} \|g_j\|^{q/p'}_{\lqcomp} \|\psi\|_{\lrpcomp}^{r'/p'},
\end{align*}
Then, we can conclude that
\begin{equation}\label{Jzero}
J_0 \leq \CijQp  \|f_i\|_{\lpcomp}  \|g_j\|_{\lqcomp}  \|\psi\|_{\lrpcomp},
\end{equation} 
where $\CijQp$ is \eqref{Cij Q+}.\\
\textbf{Step 2. (Estimate of $J_\lij$)}. Using the estimate \eqref{Jzero} for the functions $f_i, \, g_j | \cdot|^{\lij}$ and for $f_i| \cdot|^{\lij}, \, g_j $ then \eqref{J def with weight} becomes
\begin{equation*}
J_{\lij} \leq 2^{\lij} \CijQp \left( \|f_i |\cdot|^{\lij}\|_p \, \|g_j\|_q  + \|f_i\|_p \, \|g_j | \cdot|^{\lij}\|_q \right)  \| \psi \|_{r'} .
\end{equation*}
Then, using \eqref{estimate with lij},  we obtain the final estimate for $J_\lij$,
\begin{equation*}
J_{\lij} \leq  2^{\lij} \CijQp \Mij \|f_i\|_{L^p_{\lij, i}} \|g_j\|_{L^q_{\lij, j}} \| \psi \|_{r'},
\end{equation*}
where $\Mij$ is from \eqref{Mij}.
\\
\textbf{Step 3. (Estimate of $\left\| \QFGp\right\|_{\lrnula}$)}.  Thanks to the last estimate on $J_\lij$, by duality we obtain 
\begin{equation}\label{Qij estimate}
\|Q^+_{ij} (f_i, g_j)\|_{\lrcomp} \leq 2^{\lij} \CijQp \Mij \|f_i\|_{L^p_{\lij, i}} \|g_j\|_{L^q_{\lij, j}} .
\end{equation}
Furthermore, we can estimate the norm of the vector form collision operator in terms of its components,
\begin{equation}\label{Q vector in terms of Q comp r norm}
\left\| \QFGp\right\|_{\lrnula} = \left( \sum_{i=1}^I  \int_{\mathbb{R}^N}  \left| \sum_{j=1}^I Q^+_{ij} (f_i, g_j) (v)\right|^r \md v \right)^{1/r} \!\!\!\!\!
\leq  I^\frac{r-1}{r} \sum_{i=1}^I \sum_{j=1}^I \|Q^+_{ij} (f_i, g_j)\|_{\lrcomp},
\end{equation}
which yields, using \eqref{Qij estimate},
\begin{equation*}
\left\| \QFGp\right\|_{\lrnula} 
\leq \CopQr \sum_{i=1}^I \sum_{j=1}^I  \|f_i\|_{L^p_{\lij, i}} \|g_j\|_{L^q_{\lij, j}} \leq \CopQr \sum_{i=1}^I \sum_{j=1}^I  \|f_i\|_{L^p_{\lambda, i}} \|g_j\|_{L^q_{\lambda, j}},
\end{equation*}
by monotonicity of norms since $\lambda=\max_{1 \leq i,j \leq I} \lij$. This concludes the estimate  \eqref{Q+ r norm} in this subcase.\\
\textbf{Subcase 2. ($(p,q,r) = (1,1,1)$)} \\
\textbf{Step 1. (Estimate of $J_0$)}. For this choice of indexes, $J_0$ can be estimated with 
\begin{multline*}
J_0= \int_{\mathbb{R}^N}\int_{\mathbb{R}^N} f_i(v) g_j(v-u) \mathscr{P}_{ij}(\tau_{-v} (\mathcal{R}\psi), 1)(u) \md u \md v\\
\leq \|\mathscr{P}_{ij}(\tau_{-v} (\mathcal{R}\psi), 1)\|_{\infty} \|f_{i}\|_{1} \|g_j\|_{1}  \leq \CopPinf  \|f_{i}\|_{1} \|g_j\|_{1} \|\psi\|_{\infty},
\end{multline*}
where the constant $\CopPinf$ is from \eqref{Cij from L inf norm of P}.\\
\textbf{Step 2. (Estimate of $J_\lij$)}. Using the same idea as in the previous subcase, we apply the   inequality above to  the functions $f_i, \, g_j | \cdot|^{\lij}$ and to $f_i| \cdot|^{\lij}, \, g_j$, and then using \eqref{estimate with lij} we get the estimate for $J_\lij$,
\begin{equation*}
J_{\lij} \leq  2^{\lij} \CopPinf   \Mij \|f_i\|_{L^1_{\lij, i}} \|g_j\|_{L^1_{\lij, j}} \| \psi \|_{\infty}. 
\end{equation*}
\textbf{Step 3. (Estimate of $\left\| \QFGp\right\|_{\lrnula}$)}. Writing the norm of the vector value collision operator in terms of norms of its components and exploiting duality arguments, we have
\begin{multline*}
\left\| \QFGp\right\|_{\lnula} = \sum_{i=1}^I  \int_{\mathbb{R}^N}  \left| \sum_{j=1}^I Q^+_{ij} (f_i, g_j) (v)\right| \md v 
 \\ \leq \CopQone \sum_{i=1}^I\sum_{j=1}^I \|f_i\|_{L^1_{\lij, i}} \|g_j\|_{L^1_{\lij, j}} \leq  \CopQone \sum_{i=1}^I\sum_{j=1}^I \|f_i\|_{L^1_{\lambda, i}} \|g_j\|_{L^1_{\lambda, j}},
\end{multline*}
where the constant $\CopQone$ from \eqref{C1 and Cinfty for Q+}, and 
the last inequality follows from the monotonicity of norms. 
\\
\textbf{Subcase 3. ($(p,q,r) = (p,p',\infty)$)}.\\
\textbf{Step 1. (Estimate of $J_0$)}.   If additionally 
 $(p,p')=(\infty,1)$, then $J_0$ from \eqref{J def} can be rewritten and estimated using the same ideas as in the Subcase 1,
\begin{multline*}
J_0= \int_{\mathbb{R}^N}\int_{\mathbb{R}^N} f_i(v) g_j(v-u) \mathscr{P}_{ij}(\tau_{-v} (\mathcal{R}\psi), 1)(u) \md u \md v\\
\leq \|f_{i}\|_{\infty} \left( \int_{\mathbb{R}^N}\int_{\mathbb{R}^N}  g_j(v) \mathscr{P}_{ij}(1,\tau_{v} (\mathcal{R}\psi))(u) \md u \md v \right)
\leq \Cijfive  \|f_{i}\|_{\linfcomp} \|g_j\|_{\lonecomp} \|\psi\|_{\lonecomp}.
\end{multline*}
Similarly, for $(p,p')=(1, \infty)$, we have
\begin{equation*}
J_0 \leq \Cijsix \|f_{i}\|_{\lonecomp} \|g_j\|_{\linfcomp} \|\psi\|_{\lonecomp}.
\end{equation*}
If $(p,p')\notin \{(1,\infty), (\infty,1)\}$, we rewrite $J$ as
\begin{multline*}
J_0=\int_{\mathbb{R}^N} \int_{\mathbb{R}^N}  \left( f_i(v) \mathscr{P}_{ij}(\tau_{-v} (\mathcal{R} \psi), 1)(u)^{\frac{1}{p}}\right)\\ \times
\left( g_j(v-u) \mathscr{P}_{ij}(\tau_{-v} (\mathcal{R} \psi), 1)(u)^{\frac{1}{p'}}\right)\md u \, \md v.
\end{multline*}
Then by the H\"{o}lder inequality, we obtain  
\begin{multline*}
J_0\leq \left( \int_{\mathbb{R}^N} \int_{\mathbb{R}^N}  \left( f_i(v)^p \mathscr{P}_{ij}(\tau_{-v} (\mathcal{R} \psi), 1)(u)\right) \right)^{1/p}\\ \times
\left( \int_{\mathbb{R}^N} \int_{\mathbb{R}^N}  \left( g_j(v-u)^{p'} \mathscr{P}_{ij}(\tau_{-v} (\mathcal{R} \psi), 1)(u)\right)\md u \, \md v \right)^{1/p'}.
\end{multline*}
Repeating the same procedure as in the Subcase 1, we obtain
\begin{equation}\label{J for infty}
J_0\leq  \CijQppppinf \|f_i\|_{\lpcomp}  \|g_j\|_{\lppcomp}  \|\psi\|_{\lonecomp},
\end{equation}
with a constant 
\begin{equation*}
\CijQppppinf = \Cijfour.
\end{equation*}
This notation merges all possible pairs $(p,p')$.\\
\textbf{Step 2. (Estimate of $J_\lij$)}. Applying the inequality above to the
functions $f_i, \, g_j | \cdot|^{\lij}$ and to $f_i| \cdot|^{\lij}, \, g_j$, and then using \eqref{estimate with lij} we get the estimate for $J_\lij$,
\begin{equation*}
J_{\lij} \leq  2^{\lij} \CijQppppinf  \Mij \|f_i\|_{L^p_{\lij, i}} \|g_j\|_{L^{p'}_{\lij, j}}  \| \psi \|_{1}. 
\end{equation*}
\textbf{Step 3. (Estimate of $\left\| \QFGp\right\|_{\lrnula}$)}.
Finally, from the inequality above by duality, we obtain
\begin{equation} \label{inf component wise}
\|Q^+_{ij} (f_i, g_j)\|_{\linfcomp} \leq 2^{\lij} \CijQppppinf  \Mij  \|f_i\|_{L^p_{\lij, i}} \|g_j\|_{L^{p'}_{\lij, j}},
\end{equation}
which yields
\begin{multline*}
\left\|  \QFGp \right\|_{L^\infty_{0}} = \sum_{i=1}^{I} \left\| \sum_{j=1}^I Q^+_{ij} (f_i, g_j) \right\|_{\linfcomp}
\\ \leq \CopQinf \sum_{i=1}^I\sum_{j=1}^I  \|f_i\|_{L^p_{\lij, i}} \|g_j\|_{L^{p'}_{\lij, j}}  \leq  \CopQinf \sum_{i=1}^I\sum_{j=1}^I  \|f_i\|_{L^p_{\lambda, i}} \|g_j\|_{L^{p'}_{\lambda, j}}, 
\end{multline*}
with the constant $\CopQinf$ from \eqref{C1 and Cinfty for Q+}.
\end{proof}

\subsection{Estimate of the gain operator in a polynomially weighted $L^p$ space}
To find the estimates for the positive part of the collision operator for any $r \in (1,\infty)$ we will use the estimates proved in Theorem \ref{trm Lr0 norm of Q+}, the one from Lemma \ref{lemma estimate Q 0 in Ldva 0} and Riesz-Thorin interpolation theorem.

\begin{theorem}(Gain of integrability)\label{trm qplus r fg}
	For any $\epsilon>0$, $p \in(1, \infty)$ and $k\geq 0$ the collision operator can be estimated in the following way
	\begin{multline*}
	\left\| \QFGp \right\|_{\lpk}  \leq I^{1-\frac{1}{p}} 2^{\frac{p-1}{p}} \|\F\|_{\ljedk} \left(  \epsilon^{\bar{\gamma} } \Cijroner  \| \G \|_{\lpk} 
	\right. \\  \left.
	 + \hat{C}_N\epsilon^{2+\underline{\gamma} - N}   \| \G\|^{1-\theta}_{L^1_{\frac{N-2}{1-\theta}+k}} \| \G\|^\theta_{\lpk}\right)
	\end{multline*}
with $\Cijroner$ as given in \eqref{cqplus pqr} and $\hat{C}_N$ will be given in \eqref{hat C N}.
\end{theorem}
\begin{proof}
\textbf{Step 1. Estimate of $\| \mathbb{Q}^+_{N-2}(\delta_0, \G)\|_{\lpnula}$ by Riesz Thorin Interpolation.}
We will separate first the interpolation between $L^1$ and $L^2$, and then $L^2$ and $L^\infty$.
\\
\textbf{Case 1: $p \in (1,2]$.} The first estimate we need can be proven by the Theorem \ref{trm Lr0 norm of Q+}, and is given by
\begin{equation*}
\| \mathbb{Q}^+_{N-2}(\delta_0, \G)\|_{\lnula} \leq \CopQone \| \G \|_{\lnminus}.
\end{equation*}
Now, from Lemma \ref{lemma estimate Q N-2 in Ldva 0} and the fact that $L^{\frac{2N}{2N-1}}_{N-3}\hookrightarrow L^{\frac{2N}{2N-1}}_{N-2}$ we can prove
\begin{equation*}
	\| \mathbb{Q}^+_{N-2}(\delta_0, \G) \|_{L^2_0} \leq I^{\frac{3N+1}{2N}} C_N^{\frac{1}{2}} \ \tilde{C}_N^{\frac{1}{2}} \ \| \G \|_{L^{\frac{2N}{2N-1}}_{N-2}},
\end{equation*}
with $C_N = \max_{1 \leq i,j \leq I} (C_{N, ij})$ and $C_{N, ij}$ and $\tilde{C}_N$ as given in \eqref{CNij} and \eqref{CtildaN} respectively. 

Then, by the Riesz-Thorin theorem, the interpolation of $L^1$ and $L^2$ by $L^1$ and $L^{\frac{2N}{2N-1}}$ yields
\begin{equation*}
\| \mathbb{Q}^+_{N-2}(\delta_0, \G)\|_{\lpnula} \leq 2 \left(\CopQone\right)^{1-\vartheta} \left( I^{\frac{3N+1}{2N}} C_N^{\frac{1}{2}} \ \tilde{C}_N^{\frac{1}{2}} \right)^\vartheta   \| \G \|_{L^r_{N-2}},
\end{equation*}
with the relation
\begin{equation}\label{ponetwo}
\frac{1 -1/N}{1} + \frac{1/N}{p} = \frac{1}{r}, \ \text{that is,} \ r=\frac{Np}{p(N-1)+1}, 
\end{equation}
and $\vartheta=2-2/p$.\\
\textbf{Case 2: $p \in [2, \infty)$.}
For the interpolation in this case we will use the following estimate from Theorem \ref{trm Lr0 norm of Q+}:
\begin{equation*}
\|  \mathbb{Q}^+_{N-2}(\delta_0, \G) \|_{L^\infty_0} \leq \CopQinf \ \| \G \|_{L_{N-2}^\infty}.
\end{equation*}
We use again Riesz-Thorin theorem, for the interpolation of $L^2$ and $L^
\infty$ by $L^{\frac{2N}{2N-1}}$ and $L^\infty$, to get
\begin{equation*}
\| \mathbb{Q}^+_{N-2}(\delta_0, \G) \|_{\lpnula} \leq 2 \left( I^{\frac{3N+1}{2N}} C_N^{\frac{1}{2}} \ \tilde{C}_N^{\frac{1}{2}} \right)^{1-\vartheta} \left( \CopQinf \right)^\vartheta  \ \| \G \|_{L_{N-2}^r},
\end{equation*}
where
\begin{equation}\label{ptwoinf}
r=\frac{Np}{2N-1},
\end{equation}
and $\vartheta = 1-2/p$.\\
\textbf{Step 2. Estimate of $\| \mathbb{Q}^+_{N-2}(\delta_0, \G)\|_{\lpnula}$ by H\"older's inequality.}
Consider $\theta \in (0,1)$ that relates to $r$ and $p$ in the following way
\begin{equation} \label{theta relation}
\frac{1-\theta}{1} + \frac{\theta}{p} = \frac{1}{r}.
\end{equation}
Then H\"older's inequality implies that
\begin{multline*}
\| \G \|_{L_{N-2}^r} = \left( \sum_{i=1}^I \int_{\mathbb{R}^N} \left(g_i(v) \langle v \rangle_i^{N-2} \right)^r \md v\right)^{\frac{1}{r}}\\
\leq \sum_{i=1}^I \left( \int_{\mathbb{R}^N} \left( g_i(v)  \langle v \rangle_i^{\frac{N-2}{1-\theta}} \right)^{(1-\theta)r} g_i(v)^{r\theta} \md v  \right)^{\frac{1}{r}}\\
\leq \sum_{i=1}^I  \left( \int_{\mathbb{R}^N} g_i(v)  \langle v \rangle_i^{\frac{N-2}{1-\theta}} \right)^{1-\theta} \left( \left( \int_{\mathbb{R}^N} g_i(v)^p \right)^{\frac{1}{p}} \right)^\theta \\ 
\leq I \| \G\|^{1-\theta}_{L^1_{\frac{N-2}{1-\theta}}} \| \G\|^\theta_{\lpnula}.
\end{multline*}
Therefore, we conclude that for any $r$ that satisfies \eqref{ponetwo} or \eqref{ptwoinf}, there exists $\theta \in (0,1)$ such that
\begin{equation}\label{qnmintwo}
\| \mathbb{Q}^+_{N-2}(\delta_0, \G) \|_{\lpnula} \leq \hat{C}_N \| \G\|^{1-\theta}_{L^1_{\frac{N-2}{1-\theta}}} \| \G\|^\theta_{\lpnula},
\end{equation}
where the parameter $\theta$ is defined as 
\begin{equation}\label{theta r N}
\theta := \theta_{p,N} = 
\begin{cases}
\frac{1}{N}, & \text{if} \ p\in(1,2] \ \text{and} \ r=\frac{Np}{p(N-1)+1},   \\
\frac{N(p-2)+1}{N(p-1)},  & \text{if} \ p\in[2,\infty) \ \text{and} \ r=\frac{Np}{2N-1},
\end{cases}
\end{equation}
and the constant $\hat{C}_N$ is given by
	\begin{equation}\label{hat C N}
\hat{C}_N  = 
\begin{cases}
2I \left(\CopQone\right)^{1-\vartheta} \left( I^{\frac{3N+1}{2N}} C_N^{\frac{1}{2}} \ \tilde{C}_N^{\frac{1}{2}} \right)^\vartheta, &\!\! \text{if} \ p\in(1,2],  \ r=\frac{Np}{p(N-1)+1}, \text{and} \ \vartheta=2-2/p,\\ 
2I \left( I^{\frac{3N+1}{2N}} C_N^{\frac{1}{2}} \ \tilde{C}_N^{\frac{1}{2}} \right)^{1-\vartheta} \left( \CopQinf \right)^\vartheta,  &\!\! \text{if} \ p\in[2,\infty), \,  r=\frac{Np}{2N-1}, \, \text{and} \, \vartheta = 1-2/p.
\end{cases}
\end{equation}
\textbf{Step 3. Estimate for $\| \mathbb{Q}^+_{\gamma_{ij}}(\delta_0, \G) \|_{\lpnula}$.}
Now, we will proceed as we did for the case of the $L^2$ norms. Following the idea of Proposition \ref{proposition Q+ in L20} we can show that
\begin{multline*}
\| \mathbb{Q}^+_{\gamma_{ij}}(\delta_0, \G) \|_{\lpnula} \leq 2^{\frac{p-1}{p}}  \epsilon^{\bar{\gamma}} \|\mathbb{Q}^+_{0}(\delta_0, \G) \|_{L^p} 
+ 2^{\frac{p-1}{p}}\epsilon^{2+\underline{\gamma} - N}  \|\mathbb{Q}^+_{N-2}(\delta_0, \G) \|_{\lpnula},
\end{multline*}
then using Theorem \ref{trm Lr0 norm of Q+} and \eqref{qnmintwo} we get
\begin{multline}\label{qplusgamma}
\| \mathbb{Q}^+_{\gamma_{ij}}(\delta_0, \G) \|_{\lpnula} \leq 2^{\frac{p-1}{p}}  \epsilon^{\bar{\gamma}} \Cijroner \| \G \|_{\lpnula}
+ 2^{\frac{p-1}{p}} \epsilon^{2+\underline{\gamma} - N} \hat{C}_N \| \G\|^{1-\theta}_{L^1_{\frac{N-2}{1-\theta}}} \| \G\|^\theta_{\lpnula},
\end{multline}
with $\theta$ defined as in \eqref{theta r N} and $\hat{C}_N$ as in \eqref{hat C N}. \\
\textbf{Step 4. Estimate for $\left\| \QFGp \right\|_{\lpnula} $.}
Now, from Minkowski's integral inequality we can derive the following estimate
\begin{multline*}
\left\| \QFGp \right\|_{\lpnula} 
\leq \sum_{i=1}^I \int_{ \mathbb{R}^N} f_i(x)    \left( \int_{  \mathbb{R}^N}  \, \left(\sum_{j=1}^I \tau_x Q_{ij}^+(\delta_0, \tau_{-x}g_j)(v), \right)^p \, \mathrm{d}v   \right)^{1/p} \mathrm{d}x.\\
\leq I^{1-\frac{1}{p}} \int_{ \mathbb{R}^N} \sum_{i=1}^I f_i(x) \  \| \mathbb{Q}^+(\delta_0, \tau_{-x}\G)( \cdot)\|_{\lpnula}  \mathrm{d}x.
\end{multline*}
Combining \eqref{qplusgamma} with this last estimate, it is just a matter of adding the weights as in Proposition \ref{prop ltwok qplus} to complete the proof of Theorem \ref{trm qplus r fg}.
\end{proof}

\section{Propagation of exponentially and polynomially weighted $L^p$ norms } \label{Section propagation}

\begin{proposition}\label{trm Q+}  For any $\epsilon >0, \ p \in (1, \infty), \ \gamma \in (0,1) \ \text{and} \ k\geq 0$, we have the following estimate,
\begin{multline*}
\sum_{i=1}^I \int_{  \mathbb{R}^N} f_i^{p-1} \left[ \Qp \right]_i \japvi^{pk} \, \mathrm{d}v  \\
\leq I^{2-\frac{1}{p}} 2^{\frac{p-1}{p}} \|\F\|_{\ljedk} \left(  \epsilon^{\bar{\gamma}}  \Cijroner  \| \F \|_{\lpk}^p 
+ \epsilon^{2+\underline{\gamma} - N}   \hat{C}_N  \| \F\|^{1-\theta}_{L^1_{\frac{N-2}{1-\theta}+k}} \| \F\|^{p-1+\theta}_{\lpk}\right),
\end{multline*}
where $\hat{C}_N$ depends on $\|b_{ij}\|_{\infty}$ and it is defined as in \eqref{hat C N}, and $\Cijroner$ as given in \eqref{cqplus pqr}.
\end{proposition}
\begin{proof}
The proof immediately follows from  H\"older's inequality and Theorem \ref{trm qplus r fg}.
\end{proof}


\begin{proposition} \label{prop Q} Let $\F$ satisfy the assumptions Lemma \ref{lemma lower bound}, that allows to obtain a lower bound for hard potentials transition probabilities i.e. there exists some constant $c_{lb}$ such that
\begin{equation}\label{lower bound from ExiUni}
\sum_{i=1}^I \int_{  \mathbb{R}^N} m_i f_i(t,w) \left| v- w \right|^{\gamma_{ij}} \md w \geq c_{lb} \japvj^{\bar{\gamma}}, \quad \text{for any} \ j=1,\dots,I,
\end{equation}
with $\bar{\gamma} $ defined as in \eqref{gamma bar}.

Then  the following estimate holds 
\begin{equation*}
 \sum_{i=1}^I \int_{  \mathbb{R}^N} f_i^{p-1} \left[ \Q \right]_i \japvi^{pk} \, \mathrm{d}v \leq B_k   \| \F\|^{p-1+\theta}_{\lpk}  - A_k \left\|  \F \right\|_{\lpk}^{p}.
\end{equation*}
where $A_k$ and $B_k$ are given in \eqref{const Q} and are positive for a small enough $\varepsilon>0$  and $k>k^*$ with 
\begin{equation}\label{k star}
k^*= \max \{ \bar{k}, 2+2\bar{\gamma} \},
\end{equation}
 where $\bar{k} = \max_{1\leq i,j \leq I} \{ k_*^{ij}\}$ and each $k_*^{ij}$ depends on the angular transition rate $b_{ij}$ as well as on the binary mass fraction $m_i/(m_i+m_j)$  as in \cite{IG-P-C}.
\end{proposition}

\begin{proof}
	Since we work in cut-off framework, the gain and the loss term of the collision operator can be separated
	\begin{multline*}
 \sum_{i=1}^I \int_{  \mathbb{R}^N} f_i^{p-1} \left[ \Q \right]_i \japvi^{pk} \, \mathrm{d}v \\ =	\sum_{i=1}^I \int_{  \mathbb{R}^N} f_i^{p-1} \left[ \Qp \right]_i \japvi^{pk} \, \mathrm{d}v - \sum_{i=1}^I \int_{  \mathbb{R}^N} f_i^{p-1} \left[ \Qm \right]_i \japvi^{pk} \, \mathrm{d}v
	\end{multline*}
	
	To operate on the gain part of the collision operator, and since Proposition \ref{trm Q+}  works for $b_{ij} \in L^\infty(\mathbb{S}^{N-1})$, we will split the angular part $b_{ij}$ of the transition probability terms \eqref{cross section} as follows
\begin{equation}\label{b splitting}
b_{ij}(y) = b_{ij}^1(y) + b_{ij}^\infty(y),
\end{equation}
where $b_{ij}^\infty \in L^\infty(\mathbb{S}^{N-1})$ and $\| b_{ij}^1\|_1 \leq \epsilon^{\bar{\gamma}}$. Accordingly, in this Section we introduce a notation for the gain operator so that the splitting becomes more visible. Namely, 
\begin{equation*}
 \mathbb{Q}^+_{q}(\F,\G)(v)
\end{equation*}
will stand for the gain operator \eqref{gain term vector} and \eqref{collision operator gain} choosing \eqref{cross section}, i.e.
\begin{multline}\label{Qq}
\mathbb{Q}^+_{q}(\F,\G)(v) = \left[ \sum_{j=1}^I  \int_{\mathbb{R}^N} \int_{\mathbb{S}^{N-1}} f_i(v') \, g_j(v'_*) \, \left|u\right|^{\gamma_{ij}} b_{ij}^{q}(\hat{u}\cdot \sigma) \, \mathrm{d} \sigma \, \mathrm{d}v_*\right]_{1\leq i \leq I}\\ 
=  \left[\sum_{j=1}^I Q^{+}_{\gamma_{ij}, ij,q}(f_i,g_j) \right]_{1\leq i \leq I},
\end{multline}
where $q =1,\infty$. Therefore 
\begin{multline*}
\sum_{i=1}^I \int_{  \mathbb{R}^N} f_i^{p-1} \left[ \Qp \right]_i \japvi^{pk} \, \mathrm{d}v \\ = \sum_{i=1}^I \int_{  \mathbb{R}^N} f_i^{p-1} \left[ \Qpone \right]_i \japvi^{pk} \, \mathrm{d}v +\sum_{i=1}^I \int_{  \mathbb{R}^N} f_i^{p-1} \left[ \Qpinf \right]_i \japvi^{pk} \, \mathrm{d}v
\end{multline*}
Observe that for the second term we can apply Proposition \ref{trm Q+}, while for the first term we will need an extra computation.\\
First, we invoke the following point-wise estimate for the relative velocity by the mixture brackets
\begin{equation}\label{control relative velocity}
|u| \leq 2 |v-v'_*| \leq 2 \frac{\sum_{k=1}^I m_k}{ \sqrt{m_i} \sqrt{m_j}} \japvi \langle v'_* \rangle_{j}.
\end{equation}
For the first inequality we refer to the proof of Theorem 2.1 in \cite{IG-Alonso-Taskovic}, since it is independent from the system,  and for the second one just note that
\begin{multline*}
\frac{\sqrt{m_i}}{\sum_{k=1}^I m_k} \frac{\sqrt{m_j}}{\sum_{k=1}^I m_k}| v- v'_*| \leq \min \left\{\frac{\sqrt{m_i}}{\sum_{k=1}^I m_k} , \frac{\sqrt{m_j}}{\sum_{k=1}^I m_k} \right\}| v- v'_*|  \\
\leq \frac{\sqrt{m_i}}{\sum_{k=1}^I m_k} |v| +  \frac{\sqrt{m_j}}{\sum_{k=1}^I m_k} |v'_*|\\
\leq \japvi \langle v'_* \rangle_{j}.
\end{multline*}
Hence, considering \eqref{control relative velocity} and the conservation of kinetic energy, we have the following estimate
\begin{equation*}
|u|^{\gamma_{ij}} \langle v \rangle_i^{- \frac{\gamma_{ij}}{p'}} \leq \left( 2 \frac{\sum_{k=1}^I m_k}{ \sqrt{m_i} \sqrt{m_j}} \right)^{\gamma_{ij}} \langle v' \rangle_{i}^{\frac{\gamma_{ij}}{p}} \langle v'_* \rangle_{j}^{\gamma_{ij} \left(1+\frac{1}{p}\right)}.
\end{equation*}
Therefore, using this estimate, H\"older's inequality and Theorem \ref{trm Lr0 norm of Q+},
\begin{multline*}
\sum_{i=1}^I \int_{  \mathbb{R}^N} f_i^{p-1} \left[ \Qpone \right]_i \japvi^{pk} \, \mathrm{d}v 
\leq \\ \leq \left( 2 \frac{\sum_{k=1}^I m_k}{ \min_{1\leq i \leq I}{m_i}} \right)^{\bar{\gamma}} \sum_{i=1}^I \sum_{j=1}^I \int_{\mathbb{R}^N} \int_{\mathbb{R}^N} \int_{\mathbb{S}^{N-1}} f_i(v')  \langle v' \rangle_{i}^{k+\frac{\gamma_{ij}}{p}}  f_j(v'_*) \langle v'_* \rangle_{j}^{k+\gamma_{ij} \left(1+\frac{1}{p}\right)} \\ \times b_{ij}^1(\hat{u} \cdot \sigma) \left( f_i(v) \langle v \rangle_i^{k+ \frac{\gamma_{ij}}{p}} \right)^{p-1} \md \sigma \md v_* \md v\\
\leq \left( 2 \frac{\sum_{k=1}^I m_k}{ \min_{1\leq i \leq I}{m_i}} \right)^{\bar{\gamma}} \sum_{i=1}^I \left\| \sum_{j=1}^I Q^{+}_{0, ij,1} \left(\langle \cdot \rangle_i^{k+\frac{\bar{\gamma}}{p}}f_i,\langle \cdot \rangle_j^{k+ \bar{\gamma}\left(1+\frac{1}{p} \right)} f_j \right) \right\|_p \|f_i\|_{L^p_{k + \frac{\bar{\gamma}}{p},i}}^{p-1}\\
\leq \left( 2 \frac{\sum_{k=1}^I m_k}{ \min_{1\leq i \leq I}{m_i}} \right)^{\bar{\gamma}} C_{p,1,p}^{\mathbb{Q}^+} \| \F \|_{L^1_{k + \bar{\gamma}\left( 1 + \frac{1}{p}\right)}} \| \F\|_{L^p_{k + \frac{\bar{\gamma}}{p}}}^p.
\end{multline*}
Now observe that from the definition of $C_{p,1,p}^{\mathbb{Q}^+}$ in \eqref{cqplus pqr}, the fact that $b_{ij}$ is symmetrized, and since $\| b_{ij}^1\|_1 \leq \epsilon^{\bar{\gamma}}$
\begin{multline*}
C_{p,1,p}^{\mathbb{Q}^+} \leq I^{\frac{p-1}{p}} |\mathbb{S}^{N-2}| \\
\times \max_{1 \leq i,j \leq I} \left( 2^{\gamma_{ij}} M_{ij} \max_{0 \leq s\leq 1} \left(r_{ij}(1-r_{ij})\left( \frac{1}{r_{ij}(1-r_{ij})}-2+2s\right)\right)^{-\frac{N}{2p'}}\right) \epsilon^{\bar{\gamma}}\\
:= \mathfrak{C} \epsilon^{\bar{\gamma}}
\end{multline*}

Therefore, and by monotonicity of norms,
\begin{multline*}
	\sum_{i=1}^I \int_{  \mathbb{R}^N} f_i^{p-1} \left[ \Q \right]_i \japvi^{pk} \, \mathrm{d}v \\ \leq \left( 2 \frac{\sum_{k=1}^I m_k}{ \min_{1\leq i \leq I}{m_i}} \right)^{\bar{\gamma}}  \mathfrak{C} \epsilon^{\bar{\gamma}} \| \F \|_{L^1_{k + \bar{\gamma}\left( 1 + \frac{1}{p}\right)}} \| \F\|_{L^p_{k + \frac{\bar{\gamma}}{p}}}^p+ \epsilon^{\bar{\gamma}}   I^{2-\frac{1}{p}} 2^{\frac{p-1}{p}}    \Cijroner \|\F\|_{\ljedk}  \| \F \|_{\lpk}^p \\
+ \epsilon^{2+\underline{\gamma} - N}   I^{2-\frac{1}{p}} 2^{\frac{p-1}{p}}  \hat{C}_N \|\F\|_{\ljedk} \| \F\|^{1-\theta}_{L^1_{\frac{N-2}{1-\theta}+k}} \| \F\|^{p-1+\theta}_{\lpk}\\
-  \frac{\min_{1\leq i, j \leq I}\|b_{ij}\|_{L^1(\mathbb{S}^{N-1})}}{\max_{1 \leq j \leq I} m_j} c_{l b}  \left\| \F  \right\|^p_{L^p_{k+\bar{\gamma}/p}}\\
\leq B_k   \| \F\|^{p-1+\theta}_{\lpk}  - A_k \left\|  \F \right\|_{L^p_{k + \frac{\bar{\gamma}}{p}}}^{p}
\leq B_k   \| \F\|^{p-1+\theta}_{\lpk}  - A_k \left\|  \F \right\|_{\lpk}^{p}.
	\end{multline*}
with $B_k$ and $A_k$ as defined in \eqref{const Q} and $\epsilon$ small enough such that $A_k$ is positive. Notice that by propagation of moments, $B_k$ only depends on the initial data.

\end{proof}

As a consequence of this Proposition, we are able to state the following result.

\begin{corollary}
If $\F$ is a solution of the Boltzmann system \eqref{BE}, then
\begin{equation}\label{h 1}
\frac{1}{p} \partial_t \nFlpk^p  \leq B_k   \| \F\|^{p-1+\theta}_{\lpk}   - A_k  \left\| \F  \right\|^p_{L^p_{k}},
\end{equation}
for any $k\geq  k^*$, $1\leq p < \infty$, $k^*$ as defined in \eqref{k star}, and $\theta$ as given in \eqref{theta r N}.
\end{corollary}

\

We are now able to prove the main theorems regarding the propagation of $L^p$ norms with polynomial and exponential weights, stated in section \ref{Preliminaries}.

\

\textit{Proof of theorem \ref{Gen and prop Lp norms}.}
We start by the inequality \eqref{h 1}, and we associate it to the ODE of Bernoulli type
\begin{equation}\label{ODE}
y'(t) = b y(t)^{1-c} - a y(t),
\end{equation}
with $a,b,c>0$, whose solution will be an upper bound for $\|\F\|^p_{L^p_k}$. Indeed, we can explicit solve \eqref{ODE} to get to the solution
\begin{equation*}
y(t)=\left(\frac{b}{a}\left(1-e^{-act}\right)+y(0)^{c}e^{-act}\right)^{\frac{1}{c}}
\end{equation*}
Now, we can follow the proof given in \cite{IG-P-C},Theorem 2.6, choosing $y(t)=\|\F\|^p_{L^p_k}$, $b=pB_k$, $a=p A_k$, and $c=\frac{1-\theta}{p}$ to conclude our result.
\qed

\

\textit{Proof of theorem \ref{theorem exp Lp}.}
 Let's note that from theorem 2.7 in \cite{IG-P-C} there exists a constant $0<\tilde{\alpha} \leq \alpha_0$ and a constant $\tilde{C}^e >0$  such that
\begin{equation}\label{bound for exp moment}
\expLoneF \leq \tilde{C}^e.
\end{equation}
On the other hand, for any $s \in (0,1]$ and from the conservation of kinetic energy,
\begin{equation*}
\langle v \rangle_i^s \leq \langle v' \rangle_i^s+\langle v'_* \rangle_j^s.  
\end{equation*}
Therefore
\begin{equation*}
\sum_{j=1}^I Q_{ij}(f_i,f_j)(v) e^{\alpha\langle v \rangle_i^s} \leq \sum_{j=1}^I Q^+_{ij}\left(f_i e^{\alpha\langle \cdot \rangle_i^s}, f_j e^{\alpha\langle \cdot \rangle_j^s}\right)(v) - \sum_{j=1}^I Q^-_{ij}(f_i,f_j)(v) e^{\alpha\langle v \rangle_i^s}.
\end{equation*}
Now, denoting $g_i(\cdot)=f_i(\cdot) e^{\alpha\langle \cdot \rangle_i^s}$, and using the lower bound \eqref{lower bound from ExiUni},
\begin{equation}\label{estimate for dt g}
\partial_t g_i (v) \leq \sum_{j=1}^I Q^+_{ij}(g_i,g_j)(v) - c_{lb} \frac{min_{1\leq i,j \leq I} \|b_{ij}\|_{L^1(\mathbb{S}^{n_1})}}{\max_{1 \leq  j \leq I} m_j} g_i(v) \langle v\rangle_{i}^{\bar{\gamma}}.
\end{equation}
We can continue now, as done in Proposition \ref{trm Q+} and Proposition \ref{prop Q} to conclude as in Theorem \ref{Gen and prop Lp norms} that
\begin{equation*}
\| \G \|^p_{\lpnula} \leq \max \left\{\left( \frac{\hat{B_0}}{\hat{A_0}}\right)^{-\frac{p}{1-\theta}}, \|\G_0\|_{\lpnula}^p\right\},
\end{equation*}
with $\hat{B_0}$ and $\hat{A_0}$ as given in \eqref{const Bzero Azero}. Let's note that there exists $\zeta>0$ such that
\begin{equation*}
\langle \cdot \rangle_i^{\frac{N-2}{1-\theta}} \leq e^{\alpha \langle \cdot \rangle_i^{\zeta/2}}
\end{equation*} 
then,
\begin{multline*}
 \| \F e^{\alpha \langle \cdot \rangle^s} \|_{L^1_{\frac{N-2}{1-\theta}}} =  \| \F e^{\alpha \langle \cdot \rangle^s} \langle \cdot \rangle^{\frac{N-2}{1-\theta}} \|_{L^1_{0}}
 \leq \| \F e^{\alpha \langle \cdot \rangle^{s-\zeta/2}}  \|_{L^1_{0}}\leq   \| \F e^{\alpha \langle \cdot \rangle^s}  \|_{L^1_{0}}.
\end{multline*}
Therefore, by \eqref{bound for exp moment}, this upper estimate will be finite for $\alpha < \tilde{\alpha}$ and the inequality \eqref{prop exp weight} yields.
\qed

\

\


\section{$L^\infty$ estimates} \label{Section Linfty}

Our last goal in this work is to extend the propagation of $L^p_\beta$ norms to the case of $p=\infty$.
As in \cite{IG-Alonso-Taskovic}, equation (3), for technical reasons, we separate the angular part $b$ of the transition probability terms \eqref{cross section} as follows
\begin{equation}\label{b splitting bis}
b_{ij}(y) = b_{ij}(y) \left(  \mathds{1}_{y \leq \sqrt{1-\varepsilon_{ij}^2} }  +  \mathds{1}_{y > \sqrt{1-\varepsilon_{ij}^2}}\right)=: \bone(y) +\btwo(y)  ,
\end{equation}
 and, as before, the support of $b_{ij}$ is assumed to be in $[0,1]$, because of the symmetrization assumption. As in \eqref{Qq}, we use a notation for the gain operator so that the splitting becomes more visible. Namely, 
\begin{multline*}
\mathbb{Q}^+_{\varepsilon_{ij},q}(\F,\G)(v) = \left[ \sum_{j=1}^I  \int_{\mathbb{R}^N} \int_{S^{N-1}} f_i(v') \, g_j(v'_*) \, \left|u\right|^\gamma_{ij} b_{ij}^{\varepsilon_{ij},q}(\hat{u}\cdot \sigma) \, \mathrm{d} \sigma \, \mathrm{d}v_*\right]_{1\leq i \leq I}\\ 
=  \left[\sum_{j=1}^I Q^{+, \varepsilon_{ij}}_{\gamma_{ij, ij,q}}(f_i,g_j) \right]_{1\leq i \leq I},
\end{multline*}
where $q =1,2$.

\begin{lemma}\label{lemma Q+ with separated b}
	Let the transition probability terms $\mathcal{B}_{ij}$ be given with \eqref{cross section} and \eqref{b splitting bis}. Then, when $\gamma_{ij}=0$ for all $1\leq i \leq I$, the following inequalities hold
\begin{equation}\label{split for Q+}
\begin{split}
\left\| \mathbb{Q}^+_{\varepsilon_{ij},1}(\F,\G) \right\|_{\linfnula} &\leq C^{b_1}  \left\| \F \right\|_{L^2_0}  \| \G \|_{L^2_0},
\\
\left\| \mathbb{Q}^+_{\varepsilon_{ij},2}(\F,\G) \right\|_{\linfnula} &\leq  C^{b_2} \left\|\F\right\|_{L^\infty_0} \|\G\|_{L^1_0},
\end{split}
\end{equation}
where the constant  $C^{b_1} \sim  \max_{1 \leq i,j \leq I}  \left( \varepsilon_{ij}^{-\frac{N}{2}} \bbjed \right)$, and \\ $C^{b_2} \sim  \max_{1 \leq i,j \leq I}  \|\btwo\|_{L^1(\mathbb{S}^{N-1}, \md \sigma)}$. In particular, $C^{b_2} $ decreases with $\varepsilon_{ij} \searrow 0$. Therefore,
\begin{equation}\label{c2 control}
C^{b_2} \leq \xi,
\end{equation}
where $\xi$ can be taken $\xi < 1$ arbitrarily.
\end{lemma}
\begin{proof}
	For the first estimate, we apply the result of Theorem \ref{trm Lr0 norm of Q+} for the case $(p,q,r)=(2,2,\infty)$, which yields the desired estimate with a constant 
	\begin{multline*}
\CopQinf  = \max_{1 \leq i,j \leq I} \Mij \left|S^{N-2}\right|  \left(\int_{0}^1 \left(2(1-r_{ij})^2(1-s)\right)^{-\frac{N}{2}}  \bone(s) (1-s^2)^{\frac{N-3}{2}} \md s \right)^{1/2} \\ \times \left(\int_{0}^1 \left(r_{ij}(1-r_{ij})\left( \frac{1}{r_{ij}(1-r_{ij})} -2 +2s \right)\right)^{-\frac{N}{2}} \bone(s) (1-s^2)^{\frac{N-3}{2}} \md s \right)^{1/2}.
\end{multline*}
We can estimate this constant using the inequality $(1-s)^{-1}\leq 2 \varepsilon_{ij}^{-2}$ in the support of $\bone$, and estimate $\frac{1}{r_{ij}(1-r_{ij})} -2 +2s \geq \frac{1}{r_{ij}(1-r_{ij})} -2$ for $s\geq 0$,
	\begin{multline*}
\CopQinf  \leq  \max_{1 \leq i,j \leq I} \Mij  \left|S^{N-2}\right|  \  r_{ij}^{-\frac{N}{4}} \left(1-r_{ij}\right)^{-\frac{3N}{4}} \left( \frac{1}{r_{ij}(1-r_{ij})} -2 \right)^{-\frac{N}{4}} \\ \times \varepsilon_{ij}^{-\frac{N}{2}} \left(\int_{0}^1 b_{ij}(s) (1-s^2)^{\frac{n-3}{2}} \md s \right)
\\
=  \max_{1 \leq i,j \leq I} \Mij  r_{ij}^{-\frac{N}{4}} \left(1-r_{ij}\right)^{-\frac{3N}{4}} \left( \frac{1}{r_{ij}(1-r_{ij})} -2  \right)^{-\frac{N}{4}} \varepsilon_{ij}^{-\frac{N}{2}} \bbjed\\
:= C^{b_1}.
\end{multline*}

Next, for the second estimate we use the same Theorem \ref{trm Lr0 norm of Q+} for $(p,q,r)=(\infty,1,\infty)$. The constant then becomes
\begin{multline*}
\CopQinf =  \max_{1 \leq i,j \leq I} \Mij  |S^{N-2}|  \int_{0}^1 \left(r_{ij}(1-r_{ij})\left( \frac{1}{r_{ij}(1-r_{ij})} -2 +2s \right)\right)^{-\frac{N}{2}} \\
 \qquad \qquad \times \btwo(s) (1-s^2)^{\frac{N-3}{2}} \md s 
\\
\leq   \max_{1 \leq i,j \leq I} \Mij  \left(r_{ij}(1-r_{ij})\left( \frac{1}{r_{ij}(1-r_{ij})} -2  \right)\right)^{-\frac{N}{2}}  
 \|\btwo\|_{L^1(\mathbb{S}^{N-1}, \md \sigma)}\\
:=  C^{b_2} .
\end{multline*}
Note that $\bbdvaepsjed \rightarrow 0$, as $\varepsilon_{ij} \rightarrow0$ by the dominate convergence theorem. Then, inequalities \eqref{split for Q+} hold.
\end{proof}

\begin{theorem}\label{propagation of l inf norm}
Let $\gamma_{ij} \in (0,1]$, $b_{ij} \in L^1(\mathbb{S}^{N-1})$, and an initial data satisfying the hypothesis of Theorem 2.6 in \cite{IG-P-C} and such that
\begin{equation*}
\| \F_0 \|_{\linfnula} = \mathcal{C}_0,
\end{equation*} for some positive constant $\mathcal{C}_0$. Then there exists a constant $\mathcal{C}(\F_0)$ depending on $\bar{\gamma}, \, m_i, \, b_{ij}$ such that
\begin{equation}\label{ineq l inf norm}
\| \F(t, \cdot) \|_{\linfnula} \leq \mathcal{C}(\F_0), \quad t\geq 0,
\end{equation}
for $\F$ the solution of the Boltzmann system, with $\bar{\gamma}$ defined as in \eqref{gamma bar}.
\end{theorem}

This proof follows a different approach from the case $1 < p < \infty$, since now we need to focus on point-wise estimates on the collisional integral rather than in computing $L^p$ norms by duality.

\begin{proof}
Component-wise, we can apply inequality \eqref{control relative velocity} to the gain part of the collisional operator,
\begin{multline*}
 \sum_{j=1}^I Q^{+}_{ij} (f_i, f_j) (v)  \leq  \sum_{j=1}^I 2 \frac{\sum_{k=1}^I m_k}{ \sqrt{m_i} \sqrt{m_j}} Q^{+}_{0,ij} (f_i, f_j \langle \cdot \rangle_j^{\gamma_{ij}}) (v)  \japvi^{\gamma_{ij}}\\
 \leq  2  \frac{\sum_{k=1}^I m_k}{ \min_{1\leq i \leq I} \sqrt{m_i} \min_{1\leq j \leq I}{\sqrt{m_j}}}   \sum_{j=1}^I Q^{+}_{0,ij} (f_i, f_j \langle \cdot \rangle_j^{\bar{\gamma}}) (v)  \japvi^{\bar{\gamma}}  
\end{multline*}
For the loss term, from Lemma \ref{lemma lower bound}
\begin{multline*}
 \sum_{j=1}^I Q^-_{ij} (f_i, f_j) (v) = f_i(v) \sum_{j=1}^I \|b_{ij}\|_{L^1(\mathbb{S}^{N-1})}   \int_{\mathbb{R}^N}  f_j(v_*) |v-v_{*}|^{\gamma_{ij}}  \md \sigma \md v_*\\
 \geq  f_i(v)  \min_{1\leq i, j \leq I} \|b_{ij}\|_{L^1(\mathbb{S}^{N-1})} \sum_{j=1}^I    \int_{\mathbb{R}^N}  f_j(v_*) |v-v_{*}|^{\gamma_{ij}}  \md v_*\\
 \geq  \min_{1\leq i, j \leq I} \|b_{ij}\|_{L^1(\mathbb{S}^{N-1})} \frac{1}{\max_{1\leq j \leq I}m_j} c_{lb} f_i(v)  \japvi^{\bar{\gamma}}\\
\end{multline*}
where $c_{lb}$ is given by \eqref{clb}.

Denoting
\begin{equation*}
C_G := 2  \frac{\sum_{k=1}^I m_k}{ \min_{1\leq i \leq I} \sqrt{m_i} \min_{1\leq j \leq I}{\sqrt{m_j}}} \ \ \text{and} \ \ 
C_L :=   \frac{\min_{1\leq i, j \leq I} \|b_{ij}\|_{L^1(\mathbb{S}^{N-1})} }{\max_{1\leq j \leq I}m_j}.
\end{equation*}
Then, we obtain the following control
\begin{equation*}
\partial_t f_i = \sum_{j =1}^I Q_{ij}(f_i,f_j)(v) \leq C_G \sum_{j=1}^I Q^{+}_{0, ij} (f_i, f_j \langle \cdot \rangle_j^{\bar{\gamma}}) (v)  \japvi^{\bar{\gamma}} - C_L c_{lb} f_i(v) \japvi^{\bar{\gamma}}, 
\end{equation*}
or equivalently,
\begin{equation*}
\partial_t f_i + C_L c_{lb} f_i(v) \japvi^{\bar{\gamma}} \leq C_G \sum_{j=1}^I Q^{+}_{0, ij} (f_i, f_j \langle \cdot \rangle_j^{\bar{\gamma}}) (v)  \japvi^{\bar{\gamma}}. 
\end{equation*}
Now, we can split the angular function as in \eqref{b splitting bis} and apply Lemma \ref{lemma Q+ with separated b} component wise, as showed in \eqref{inf component wise}
\begin{multline*}
\partial_t \left( f_i(v) e^{C_L c_{lb} \japvi^{\bar{\gamma}}t} \right) \leq e^{C_L c_{lb} \japvi^{\bar{\gamma}}t} C_G \left( \sum_{j=1}^I  Q^{+}_{0,ij} (f_i, f_j \langle \cdot \rangle_j^{\bar{\gamma}}) (v)  \right) \japvi^{\bar{\gamma}} \\
\leq e^{C_L c_{lb} \japvi^{\bar{\gamma}}t} C_G \left( \sum_{j=1}^I  \|Q^{+}_{0,ij} (f_i, f_j \langle \cdot \rangle_j^{\bar{\gamma}})\|_{\infty}  \right) \japvi^{\bar{\gamma}}\\
\leq e^{C_L c_{lb} \japvi^{\bar{\gamma}}t} C_G  \sum_{j=1}^I \left(  \|Q^{+}_{0, ij,\bone} (f_i, f_j \langle \cdot \rangle_j^{\bar{\gamma}})\|_{\infty} + \|Q^{+}_{0, ij,\btwo} (f_i, f_j \langle \cdot \rangle_j^{\bar{\gamma}})\|_{\infty}  \right) \japvi^{\bar{\gamma}}\\
\leq e^{C_L c_{lb} \japvi^{\bar{\gamma}}t} C_G   \left( C^{b_1} \sum_{j=1}^I \|f_i\|_2 \|f_j \langle \cdot \rangle_j^{\bar{\gamma}}\|_2 +  C^{b_2} \sum_{j=1}^I \|f_i\|_\infty \|f_j \langle \cdot \rangle_j^{\bar{\gamma}}\|_1  \right) \japvi^{\bar{\gamma}}\\
\leq e^{C_L c_{lb} \japvi^{\bar{\gamma}}t} C_G   \left(  C^{b_1} \|f_i\|_2 \sum_{j=1}^I  \|f_j \langle \cdot \rangle_j^{\bar{\gamma}}\|_2 +  C^{b_2}\|f_i\|_\infty \sum_{j=1}^I  \|f_j \langle \cdot \rangle_j^{\bar{\gamma}}\|_1  \right) \japvi^{\bar{\gamma}}\\
\leq e^{C_L c_{lb} \japvi^{\bar{\gamma}}t} C_G   \left( C^{b_1} \|f_i\|_2  \| \F \|_{L^2_{\bar{\gamma}}} +   C^{b_2}\|f_i\|_\infty  \| \F \|_{L^1_{\bar{\gamma}}} \right) \japvi^{\bar{\gamma}}.
\end{multline*}
Now, observe that from \eqref{k star}, $\| \F \|_{L^1_{\bar{\gamma}}}  \leq \| \F \|_{L^1_{k_*}}$ and since $\F_0 \in \Omega$ then by theorem 2.6 in \cite{IG-P-C}, $ \| \F \|_{L^1_{\bar{\gamma}}} \leq C(\F_0)$, with $C(\F_0)$ as given in that theorem. 
In the same way, $\| \F \|_{L^2_{\bar{\gamma}}} \leq \| \F \|_{L^2_{k_*}}$ and by theorem \ref{Gen and prop Lp norms} there is propagation of $L^2$ norms, so $ \| \F \|_{L^2_{\bar{\gamma}}} \leq \tilde{C}(\F_0)$, with $\tilde{C}(\F_0)$ as given in \eqref{propagation of Lp} .

Moreover, we can choose $\max_{1 \leq i, j \leq I} \epsilon_{ij}$ small enough such that, choosing $\xi =   \frac{c_{lb}C_L}{4 C(\F_0) C_GI}$ in \eqref{c2 control},
\begin{equation*}
  C^{b_2} \leq \frac{c_{lb}C_L}{4 C(\F_0) C_GI}.
\end{equation*}
Then, it follows that
\begin{equation*}
\partial_t \left( f_i(v) e^{C_L c_{lb} \japvi^{\bar{\gamma}}t} \right) \leq e^{C_L c_{lb} \japvi^{\bar{\gamma}}t}  \left(C_G  C^{b_1}  \tilde{C}(\F_0) \| f_i \|_2 + \frac{c_{lb}C_L}{4I} \| f_i \|_{\infty} \right) \japvi^{\bar{\gamma}}.
\end{equation*}
Now, doing integration over $t$
\begin{multline*}
f_i(v) \leq f_i(0,v) e^{-C_L c_{lb} \japvi^{\bar{\gamma}}t} + \\
\ \ \ \int_{0}^t e^{-C_L c_{lb} \japvi^{\bar{\gamma}}(t-s)}   \left(C_G  C^{b_1} \tilde{C}(\F_0) \|f_i\|_2  +  \frac{c_{lb}C_L}{4I } \|f_i\|_\infty  \right) \md s \japvi^{\bar{\gamma}}\\
\leq \|f_i(0, \cdot) \|_\infty +\frac{ C_GC^{b_1} \tilde{C}(\F_0)}{ c_{lb}C_L} \sup_{0 \leq s\leq t} \|f_i(s, \cdot) \|_2 + \frac{1}{4 I} \sup_{0 \leq s \leq t} \|f_i(s, \cdot) \|_{\infty}.
\end{multline*}
Finally, since the right hand side remains the same, we take the supremum over all $v \in \mathbb{R}^N$ and sum over $1 \leq i \leq I$, to obtain
\begin{multline*}
\| \F(t, \cdot) \|_{L^\infty_0} \leq \| \F_0 \|_{L^\infty_0} +  \frac{ C_GC^{b_1} \tilde{C}(\F_0)}{ c_{lb}C_L} \sum_{i=1}^I \sup_{0 \leq s\leq t} \|f_i(s, \cdot) \|_2 + \frac{1}{4 I} \sum_{i=1}^I \sup_{0 \leq s \leq t} \|f_i(s, \cdot) \|_{\infty}   \\
\leq \| \F_0 \|_{L^\infty_0} + \frac{C_G C^{b_1} \tilde{C}(\F_0)}{ c_{lb}C_L} \sum_{i=1}^I \sup_{0 \leq s\leq t} \|f_i(s, \cdot) \|_2 + \frac{1}{4 I} \sum_{i=1}^I \sup_{0 \leq s \leq t} \|\F(s, \cdot) \|_{\linfnula}  \\
\leq \| \F_0 \|_{L^\infty_0} + \frac{C_G C^{b_1} \tilde{C}(\F_0)}{ c_{lb}C_L} \sum_{i=1}^I \sup_{0 \leq s\leq t} \left(\|f_i(s, \cdot) \|_2^2\right)^{\frac{1}{2}} + \frac{1}{4 } \sup_{0 \leq s \leq t} \|\F(s, \cdot) \|_{\linfnula} \\
\leq \| \F_0 \|_{L^\infty_0} +  \frac{C_G C^{b_1} \tilde{C}(\F_0)}{ c_{lb}C_L} \sum_{i=1}^I \sup_{0 \leq s\leq t} \|\F(s, \cdot) \|_{L^2_0} + \frac{1}{4 } \sup_{0 \leq s \leq t} \|\F(s, \cdot) \|_{\linfnula}   \\
\leq \| \F_0 \|_{L^\infty_0} +\frac{C_G C^{b_1} \tilde{C}(\F_0)I}{ c_{lb}C_L} \sup_{0 \leq s\leq t} \|\F(s, \cdot) \|_{L^2_0} + \frac{1}{4 } \sup_{0 \leq s \leq t} \|\F(s, \cdot) \|_{\linfnula}  \\
\end{multline*}
Again, by the propagation of polynomially weighted $L^p$ norms \eqref{propagation of Lp}, 
\begin{equation*}
 \|\F(s, \cdot) \|_{L^2_0} \leq \tilde{\tilde{C}}(\mathbb{F}_0),
\end{equation*}
and taking the supremum over $t \in (0, T]$, estimate \eqref{ineq l inf norm} follows, and the proof of Theorem \ref{propagation of l inf norm} is complete.
\end{proof}

\textit{Proof of theorem \ref{propagation of l inf norm pol weight}}
From the conservation law of kinetic energy \eqref{CL micro},
\begin{equation*}
\japvi \leq \japvip \japvjps.
\end{equation*}
Then, it holds
\begin{equation*}
\partial_t \tilde{f}_i \leq \sum_{j =1}^I Q_{ij}(\tilde{f}_i,\tilde{f}_j)(v),
\end{equation*}
where 
$$
\tilde{f}_i(v) =  f_i(v) \japvi^k, \quad i=1,\dots,I.
$$
 Note that Lemma \ref{lemma Q+ with separated b} is also valid for $\tilde{\F}$ and $\tilde{\G}$ then we can obtain 
\begin{equation*}
\partial_t \tilde{f}_i \leq \sum_{j =1}^I Q_{ij}(\tilde{f}_i,\tilde{f}_j)(v) \leq C_G \sum_{j=1}^I Q^{+}_{0, ij} (\tilde{f}_i, \tilde{f}_j \langle \cdot \rangle_j^{\bar{\gamma}}) (v)  \japvi^{\bar{\gamma}} - C_L c_{lb} \tilde{f}_i(v) \japvi^{\bar{\gamma}}, 
\end{equation*}
as in the proof of theorem \ref{propagation of l inf norm}. So, it is just a matter of repeating the proof for $\tilde{f}_i$. In fact, we can obtain that 
\begin{multline*}
\partial_t \left(\tilde{f}_i(v) e^{C_L c_{lb} \japvi^{\bar{\gamma}}t} \right) 
\leq e^{C_L c_{lb} \japvi^{\bar{\gamma}}t} C_G   \left( C^{b_1} \|\tilde{f}_i\|_2  \| \tilde{\F} \|_{L^2_{\bar{\gamma}}} +   C^{b_2}\|\tilde{f}_i\|_\infty  \| \tilde{\F} \|_{L^1_{\bar{\gamma}}} \right) \japvi^{\bar{\gamma}}.
\end{multline*}
Now,  $\| \tilde{\F} \|_{L^1_{\bar{\gamma}}} = \| \F \|_{L^1_{\bar{\gamma}+k}} $ and since $\F_0 \in \Omega$ then by theorem 2.6 in \cite{IG-P-C}, we can bound $ \| \F \|_{L^1_{\bar{\gamma}+k}} \leq C(\F_0)$, with $C(\F_0)$ as given in that theorem for $k> k^*$. 
In the same way, $\| \tilde{\F} \|_{L^2_{\bar{\gamma}}}= \| \F \|_{L^2_{\bar{\gamma}+k}}$ and by theorem \ref{Gen and prop Lp norms} there is propagation of weighted $L^2$ norms, so $ \| \F \|_{L^2_{\bar{\gamma}+k}} \leq \tilde{C}(\F_0)$, with $\tilde{C}(\F_0)$ as given in \eqref{propagation of Lp} for $k>k^*$. 
Then we can redo the computations shown before to obtain
\begin{multline*}
\| \tilde{\F}(t, \cdot) \|_{L^\infty_0} 
\leq \| \tilde{\F}_0 \|_{L^\infty_0} +\frac{C_G C^{b_1} \tilde{C}(\F_0)I}{ c_{lb}C_L} \sup_{0 \leq s\leq t} \|\tilde{\F}(s, \cdot) \|_{L^2_0} + \frac{1}{4 } \sup_{0 \leq s \leq t} \|\tilde{\F}(s, \cdot) \|_{\linfnula}.  \\
\end{multline*}
Again, by the propagation of polynomially weighted $L^p$ norms \eqref{propagation of Lp}, 
\begin{equation*}
 \|\tilde{\F}(s, \cdot) \|_{L^2_0} \leq \tilde{\tilde{C}}(\mathbb{F}_0),
\end{equation*}
and taking the supremum over $t \in (0, T]$, estimate \eqref{ineq l inf norm pol weight} follows, and the proof of Theorem \ref{propagation of l inf norm} is complete.
\qed

\textit{Proof of theorem \ref{thm prop exp weight inf norm}}
 Recall from the proof of Theorem \ref{propagation of l inf norm}, 
\begin{equation}\label{relation Q+ and Q0 comp wise}
 \sum_{j=1}^I Q^{+}_{ij} (f_i, f_j) (v) 
 \leq  2  \frac{\sum_{k=1}^I m_k}{ \min_{1\leq i \leq I} \sqrt{m_i} \min_{1\leq j \leq I}{\sqrt{m_j}}}   \sum_{j=1}^I Q^{+}_{0,ij} (f_i, f_j \langle \cdot \rangle_j^{\bar{\gamma}}) (v)  \japvi^{\bar{\gamma}} , 
\end{equation}
then,  from \eqref{estimate for dt g}, yields the following estimate
\begin{multline*}
\partial_t g_i (v) \leq \left( 2  \frac{\sum_{k=1}^I m_k}{ \min_{1\leq i \leq I} \sqrt{m_i} \min_{1\leq j \leq I}{\sqrt{m_j}}}   \sum_{j=1}^I Q^{+}_{0,ij} (g_i, g_j \langle \cdot \rangle_j^{\bar{\gamma}}) (v) \right. \\ \left. - c_{lb} \frac{min_{1\leq i,j \leq I} \|b_{ij}\|_{L^1(\mathbb{S}^{n_1})}}{\max_{1 \leq  j \leq I} m_j} g_i(v) \right) \langle v\rangle_{i}^{\bar{\gamma}},
\end{multline*}
for $g_i(\cdot)=f_i(\cdot) e^{\alpha\langle \cdot \rangle_i^s}$. Then, by splitting the angular function as in \eqref{b splitting bis}, we can repeat the same argument as in the proof of Theorem \ref{propagation of l inf norm}, to conclude that
\begin{equation*}
\| \mathbb{G}(t, \cdot) \|_{\linfnula} \leq \mathcal{C }(\mathbb{G}_0)
\end{equation*}
Note that we have the boundedness of $\| \G \|_{L^1_{\bar{\gamma}}}$, $\| \G \|_{L^2_{\bar{\gamma}}}\|$ and $\| \G \|_{L^2_0}$ by the generation and propagation of exponentially weighted moments and $L^p$ norms.\\
%
%
%
%
\qed

\section{Acknowledgements.} 
The authors would like to thank Professor Ricardo J. Alonso for fruitful discussions on the topic.  The work of the first two authors has been partially supported by NSF grants DMS 1715515 and  RNMS (Ki-Net) DMS-1107444.  Milana Pavi\'c-\v Coli\'c acknowledges the support of Ministry of Education, Science and Technological Development, Republic of Serbia within the Project No. ON174016.   This work was initiated while Milana Pavi\'c-\v Coli\'c was a Fulbright Scholar from the Univerisity of Novi Sad, Serbia, visiting the Oden Institute of Computational Engineering and Sciences (Oden) at the University of Texas Austin co funded by a JTO Fellowship.  Oden support is also gratefully acknowledged.

\appendix

\section{Proof of Theorem \ref{Th Carleman} (Carleman representation of the gain operator)} \label{Appendix proof Carleman}
\begin{proof}
	The proof follows \cite{GambaPanfVil09} and \cite{AlonsoGamba18}. Indeed, using the fact that
	\begin{equation*}
	\int_{\mathbb{S}^{N-1}} F(\left|u\right| \sigma - u) \mathrm{d}\sigma = \frac{2}{\left|u\right|^{N-2}} \int_{ y \in \mathbb{R}^N} \delta(2 y \cdot u + \left| y \right|^2) F(y) \mathrm{d}y,
	\end{equation*}
	where $\delta$ is usual Dirac delta function, one has
	\begin{multline*}
	Q^{ +}_{ij}(f,g)(v) = \int_{v_* \in \mathbb{R}^N}  \frac{2}{\left|u\right|^{N-2}}  \int_{ y \in \mathbb{R}^N} \delta(2 y \cdot u + \left| y \right|^2) \ f(v+(1-r_{ij})y) \ g(v_*-r_{ij} y) \\ \times \mathcal{B}_{ij}\left(\left|u\right|,1-\frac{\left|y\right|^2}{2\left|u\right|^2}\right) \mathrm{d}y \, \mathrm{d}v_*,
	\end{multline*}
	where $u=v-v_*$, and $\left|u\right|=\left|v-v_*\right|= \left|v-v_* + y\right|$.
	We first perform the change of variable $y\mapsto v'=v+(1-r_{ij})y$, with the Jacobian $(1-r_{ij})^N$, and then for $v'$ fixed $v_* \mapsto v'_* = v_* - \frac{r_{ij}}{(1-r_{ij})} (v'-v)$ with the unit Jacobian. After some calculations, this gives
	\begin{multline*}
	Q^{ +}_{ij}(f,g)(v) = 2 (1-r_{ij})^{-N}\int_{v' \in \mathbb{R}^N}  \int_{ v'_* \in \mathbb{R}^N}  \left|u\right|^{-N+2} \ f(v') \ g(v'_*) \\ \times 
	\delta\left(\frac{2}{1-r_{ij}}(v'-v)  \cdot u + \frac{\left| v'-v \right|^2}{(1-r_{ij})^2}\right) 
	\\ \times \mathcal{B}_{ij}\left(\left|u\right|,1-\frac{\left|v'-v\right|^2}{2(1-t)^2\left|u\right|^2}\right) \mathrm{d}v'_* \, \mathrm{d}v',
	\end{multline*}
	where $u=\frac{1}{1-r_{ij}}v - \frac{r_{ij}}{1-r_{ij}} v' -v'_*$, and $\left|u\right|=\left|v'-v'_*\right|$.  With this notation, we can precise the argument of delta function,
	$$
	\frac{2}{1-r_{ij}}(v'-v)  \cdot u + \frac{\left| v'-v \right|^2}{(1-r_{ij})^2} =
	\frac{2}{1-r_{ij}}(v-v')  \cdot (v'_*-v) + \frac{1-2 r_{ij}}{(1-r_{ij})^2} \left| v'-v \right|^2 .
	$$
	This leads us to perform the change of variables 
	$
	v'_* \mapsto \bar{v}= v'_*-v, 
	$
	with unit Jacobian, which implies
	\begin{multline}\label{pomocna 2}
	Q^{ +}_{ij}(f,g)(v) = 2 (1-r_{ij})^{-N}\int_{v' \in \mathbb{R}^N}  \int_{ \bar{v} \in \mathbb{R}^N}  \left|v'-\bar{v}-v\right|^{-N+2} \ f(v') \ g(\bar{v}+v) \\ \times 
	\delta\left(	\frac{2}{1-r_{ij}}(v-v')  \cdot \bar{v} + \frac{1-2 r_{ij}}{(1-r_{ij})^2} \left| v'-v \right|^2 \right) 
	\\ \times \mathcal{B}_{ij}\left(\left|v'-\bar{v}-v\right|,1-\frac{\left|v'-v\right|^2}{2(1-t)^2\left|v'-\bar{v}-v\right|^2}\right) \mathrm{d}\bar{v} \, \mathrm{d}v'.
	\end{multline}
	Now, for fixed $v'$ and $v$, we decompose the vector $\bar{v}$ into the component parallel  to $v-v'$  and orthogonal to it. Let 
	$$n=\frac{v-v'}{\left|v-v'\right|}.$$ 
	Denote with $z^{\parallel}$  the component of $\bar{v}$ in direction of $n$, i.e. let $z^{\parallel} := z \cdot n $. Then its orthogonal vector $z^{\perp}$ belong to the hyperplane $\E$, by definition  \eqref{hyperplane E}. Thus, one may write 
	$$\bar{v}=z^{\perp} + z^{\parallel}n, \quad z^{\parallel} = z \cdot n, \ \ z^{\perp} \in \E,$$  from which it follows.\\
	Moreover, in \eqref{pomocna 2}, we change the variables $\bar{v}\mapsto z^{\perp} + z^{\parallel}n$ with unit Jacobian, and integration is performed via integration with respect to its components $z^{\perp} \in \E$ and $z^{\parallel} \in \mathbb{R}$. This change of variables simplifies the argument of delta function. Indeed,
	\begin{equation*}
	\frac{2}{1-r_{ij}}(v-v')  \cdot \bar{v} + \frac{1-2 r_{ij}}{(1-r_{ij})^2} \left| v'-v \right|^2
	= \frac{2 \left|v'-v\right|}{1-r_{ij}} \left( \frac{1-2r_{ij}}{2(1-r_{ij})} \left|v'-v\right| +{z}^\parallel \right),
	\end{equation*}
	which yields the following representation
	\begin{multline*}
	Q^{ +}_{ij}(f,g)(v) = 2 (1-r_{ij})^{-N}\int_{v' \in \mathbb{R}^N} f(v')  \int_{ z^\perp \in \E} \int_{{z}^\parallel \in \mathbb{R}}  \left|v'-v-z^\perp-z^\parallel n \right|^{-N+2}  \\ \times g\left(v+ z^\perp +z^\parallel n\right) \delta\left(-\frac{2 \left|v'-v\right|}{1-r_{ij}} \left( \frac{2r_{ij}-1}{2(1-r_{ij})} \left|v'-v\right| -{z}^\parallel \right)\right) \\ \times \mathcal{B}_{ij}\left(\left|v'-v-z^\perp-z^\parallel n \right|,1-\frac{\left|v'-v\right|^2}{2(1-r_{ij})^2\left|v'-v-z^\perp-z^\parallel n \right|^2}\right) \mathrm{d}{z}^\parallel\, \mathrm{d}z^\perp \, \mathrm{d}v'.
	\end{multline*}
	Then, by the fact
	\begin{equation*}
	\int_{ y \in \mathbb{R}} \delta(a(b-y)) F(y) \mathrm{d}y= \left|a\right|^{-1} F(b),
	\end{equation*}
	one obtains
	\begin{multline*}
	Q^{+}_{ij}(f,g)(v) = (1-r_{ij})^{-N+1}     \int_{v' \in \mathbb{R}^N} \frac{f(v')}{ \left|v'-v\right|}  \int_{ z^\perp\in\E}   g\left(z^\perp+\p\right) \\ \times {\left|\frac{(v-v')}{2(1-r_{ij})}+z^\perp\right|^{2-N}} \\ \times \mathcal{B}_{ij}\left(\left|\frac{(v-v')}{2(1-r_{ij})}+z^\perp\right|,1-\frac{\left|v'-v\right|^2}{2(1-r_{ij})^2\left|\frac{(v-v')}{2(1-r_{ij})}+z^\perp\right|^2}\right) \mathrm{d}z^\perp \, \mathrm{d}v'.
	\end{multline*}
	It remains to rename the variables.
\end{proof}

\section{Existence and Uniqueness} \label{existence and uniqueness}
For a closure to this work, we will transcript the theorem from \cite{IG-P-C} and the idea behind the proof.

\medskip

\begin{theorem}[Existence and Uniqueness]\label{theorem existence uniqueness} Assume that $\F(0,v)=\F_0(v) \in \Omega$, where
 \begin{multline}\label{Omega}
\Omega = \Big\{ \F(t, \cdot) \in L^1_2: \F \geq 0 \ \text{in} \ v,  \sum_{i=1}^I \int_{\mathbb{R}^N} m_i v \, f_i(t,v)  \mathrm{d}v=0,  
\\
\exists \,  c_0, C_0, c_2 ,  C_2, C_{2+\varepsilon} >0, \text{and} \ C_0 < c_2, \text{ such that} \ \forall t\geq0, 
\\
c_0 \leq \mathfrak{m}_0 F (t) \leq C_0,  \ c_2 \leq \mathfrak{m}_2 F (t) \leq C_2,   
\\
\mathfrak{m}_{2+\varepsilon} F (t) \leq C_{2+\varepsilon}, \ \text{for } \ \varepsilon>0,  
\\
\mathfrak{m}_{k_*}F (t) \leq C_{k_*}, \ \text{with } k_* \text{ as  in \eqref{k star}, and} \ C_{k_*} \ \text{as in} \text{ \eqref{c k star}} \  \Big\} 
\end{multline}
where
\begin{equation*}
\mathfrak{m}_\omega F(t) := \left\| \F \right\|_{L^1_\omega}= \sum_{i=1}^I \int_{\mathbb{R}^3} \left| f_i(t,v)  \right|\japvi^{\omega} \mathrm{d}v.
\end{equation*} 
  Then the Boltzmann system \eqref{Cauchy problem} for the cross section \eqref{cross section} has a unique solution in $\mathcal{C}\left(\left[0, \infty \right), \Omega \right) \cap \mathcal{C}^1\left(\left(0,\infty  \right), \ldva \right)$.
  \end{theorem}

The proof is based in an abstract framework of ODE theory in Banach spaces, which can be found in \cite{Martin}.
In order to apply that theory, it is crucial to state the invariant region $\Omega \subset L^1_2$ in which the collision operator $\mathbb{Q}: \Omega \rightarrow L^1_2$ satisfies \emph{(i) H\"{o}lder continuity},  \emph{(ii) Sub-tangent} and \emph{(iii) one-sided Lipschitz} conditions. 

To that end, the authors studied the  map $\mathcal{L}_{\bar{\gamma}, k_*}:[0,\infty) \rightarrow \mathbb{R}$,  defined as $$\mathcal{L}_{\bar{\gamma}, k_*}(x)=- Ax^{1+ \frac{\bar{\gamma}}{k_*}}+  B   x,$$ 
where $A$ and $B$ are positive constants,  $\bar{\gamma} \in (0,1]$ and  $k_*$ as defined in  \eqref{k star}. This map has only one root, denoted with $x^*_{\bar{\gamma}, k_*}$, at which $\mathcal{L}_{\bar{\gamma}, k_*}$ changes from positive to negative. Thus, for any $x\geq0$, they  write
\begin{equation*}
\mathcal{L}_{\bar{\gamma}, k_*}(x) \leq \max_{0\leq x \leq x^*_{\bar{\gamma}, k_*}} \mathcal{L}_{\bar{\gamma}, k_*}(x) =: \mathcal{L}^*_{\bar{\gamma}, k_*}.
\end{equation*}
Then, defining
\begin{equation}\label{c k star}
C_{k_*} := x^*_{\bar{\gamma}, k_*} +  \mathcal{L}^*_{\bar{\gamma}, k_*},
\end{equation}
they were able to write such a region $\Omega$.

\section{Lower bound of the cross section}
\label{section lower bound}
We will state the following Lemma, whose proof can be found in \cite[Appendix]{IG-P-C}.

\begin{lemma}\label{lemma lower bound}
	Let $\gij \in [0,2]$, for any $i, j \in \left\{1,\dots, I \right\}$, and assume \\ $0\leq \left\{ F(t) = \left[ f_1(t) \dots f_I(t) \right]^T \right\}_{t\geq 0} \subset \ldva$ satisfies
	\begin{multline*}
	c \leq \sum_{i=1}^I \int_{\mathbb{R}^3} m_i \, f_i (t, v)  \mathrm{d}v \leq C, \qquad c \leq \sum_{i=1}^I \int_{\mathbb{R}^3} f_i (t, v) m_i \left|v\right|^2 \mathrm{d}v \leq C,\\ \sum_{i=1}^I \int_{\mathbb{R}^3} f_i (t, v) m_i v \mathrm{d}v=0, 
	\end{multline*} 
	for some positive constants $c$ and $C$. Assume also boundedness of the moment
	\begin{equation*}
	\sum_{i=1}^I \int_{\mathbb{R}^3} f_i (t, v) m_i \left|v\right|^{2+\varepsilon}  \mathrm{d}v \leq B, \quad \varepsilon>0. 
	\end{equation*}
	Then, there exists a constant $\clb$  defined as
		\begin{multline}\label{clb}
	\clb
= \frac{c}{2} \ct  \left( 2^{2+\varepsilon} \left(\frac{\max\{C,B\}}{c} \right)\left(1+ \left( \frac{{\color{black} 2} \, C}{\ct \, c} \right)^\frac{2}{\bar{\gamma}}\right)^{\frac{2+\varepsilon}{2}}\right)^{\frac{-2+\underline{\gamma}}{\varepsilon}}
\\ { \color{black} \times \left( 1 + \frac{\max_{1\leq j \leq I} m_j}{\sum_{i=1}^I m_i}  \left( \frac{{\color{black} 2} \, C}{\ct \, c} \right)^2 \right)^{-\bar{\gamma}/2},}
 	\end{multline} such that
	\begin{equation}\label{lower bound}
	\sum_{i=1}^I \int_{\mathbb{R}^3} m_i f_i (t, w) \left|v-w\right|^\gij \md w \geq \clb \japvj^{\bar{\gamma}}, 
	\end{equation}
	for any $j \in \{1,\dots,I\}$, {\color{black} with $\bar{\gamma}=\max_{1 \leq i,j \leq I} \gij$}.
\end{lemma}

\medskip
\medskip

\end{document}